\documentclass[graybox]{svmult}  

\usepackage[numbers]{natbib}
\usepackage{amsmath}
\usepackage{amssymb}
\usepackage{psfrag}
\usepackage{epsfig}
\usepackage{mathptmx}       
\usepackage{helvet}         
\usepackage{courier}        
\usepackage{type1cm}        
\usepackage{makeidx}         
\usepackage{graphicx}        
\usepackage{multicol}        
\usepackage[bottom]{footmisc}

\makeindex             

\DeclareMathOperator{\disc}{disc}
\newcommand{\N}{\mathbb{N}}
\newcommand{\R}{\mathbb{R}}

\usepackage[algo2e,ruled,vlined,linesnumbered]{algorithm2e}
\usepackage{url} 
\usepackage{longtable} 
\usepackage{tikz}
\usepackage{pgfplots}

\newcommand{\BB}{\mathcal{B}}
\newcommand{\CC}{\mathcal{C}}
\newcommand{\II}{\mathcal{I}}
\newcommand{\NN}{\mathcal{N}}
\newcommand{\PP}{\mathcal{P}}
\newcommand{\RR}{\mathcal{R}}

\begin{document}

\title*{Calculation of Discrepancy Measures and Applications}
\author{Carola Doerr, Michael Gnewuch and Magnus Wahlstr\"om}
\institute{
Carola Doerr \at 
Universit\'e Paris Diderot - Paris 7, LIAFA, Paris, France 
and \\
Max Planck Institute for Informatics, 
Saarbr\"ucken, Germany, \email{winzen@mpi-inf.mpg.de}
\and
Michael Gnewuch \at School of Mathematics and Statistics, University of New South Wales, Sydney, NSW, 2052, Australia, \email{m.gnewuch@unsw.edu.au}
\and Magnus Wahlstr\"om \at Max Planck Institute for Informatics, 
Saarbr\"ucken, Germany, \email{wahl@mpi-inf.mpg.de}
}
%
%
\maketitle

\abstract{In this book chapter we survey known approaches and algorithms to compute discrepancy measures of 
point sets. After providing an introduction which puts the calculation of discrepancy measures in a more general
context, we focus on the geometric discrepancy measures for which computation algorithms have been designed.
In particular, we explain methods to determine $L_2$-discrepancies and approaches to tackle
the inherently difficult problem to calculate the star discrepancy of given sample sets. 
We also discuss in more detail three applications of algorithms to approximate discrepancies. 
}

\section{Introduction and Motivation}
\label{INTRO}

In many applications it is of interest to measure the quality of certain 
point sets, e.g., to test whether successive pseudo-random numbers are statistically independent, see, e.g., \cite{Knu81, LH98, Nie92}, or whether certain
sample sets are suitable for multivariate numerical integration of certain
classes of integrands, see, e.g., \cite{DP12PD}. Other areas where the need of such measurements 
may occur include
the generation of low-discrepancy samples, the design of computer experiments,  
computer graphics, and stochastic programming. (We shall describe some of these applications
in more detail in Section~\ref{APPLICATIONS}.) A particularly useful class of quality measures, on which we want to focus in this book chapter, is the class of discrepancy measures.
Several different discrepancy measures are known. 
Some of them allow for an efficient
evaluation, others are hard to evaluate in practice. We shall give several examples below, but before
doing so, let us provide a rather general definition of a geometric discrepancy measure.

Let $(M, \Sigma)$ be a measurable space. Now let us consider two measures $\mu$ and $\nu$ defined on the $\sigma$-algebra $\Sigma$ of $M$. A typical situation would be that 
$\mu$ is a rather complicated measure, e.g., a continuous measure or a discrete measure supported on a large number of atoms, and $\nu$ is a simpler measure, e.g., a
discrete probability measure with equal probability weights or  a discrete (signed) measure supported on a small number
of atoms. We are interested in approximating $\mu$ by the simpler object $\nu$ in some sense and want to quantify the approximation quality. This can be done with the help
of an appropriate discrepancy measure.

Such situations occur, e.g., in numerical integration, where one has to deal with a continuous measure $\mu$ to evaluate integrals of the 
form $\int_M f\,{\rm d}\mu$ and wants to approximate these integrals with the help of a quadrature formula
\begin{equation}\label{quadrature}
 Qf = \sum^n_{i=1} v_i f(x^{(i)}) = \int_M f \,{\rm d}\nu;
\end{equation}
here $\nu = \nu_Q$ denotes the discrete signed measure $\nu(A)= \sum^n_{i=1} v_i 1_A(x^{(i)})$, with $1_A$ being the characteristic
function of $A\in \Sigma$.
Another instructive example is scenario reduction in stochastic programming, which will be discussed in more detail in Section~\ref{SCENARIO}.

To quantify the discrepancy of $\mu$ and $\nu$ one may select a subset $\mathcal{B}$ of the $\sigma$-algebra $\Sigma$, the class of 
\emph{test sets},
to define the \emph{local discrepancy} of $\mu$ and $\nu$ in a test set 
$B \in \mathcal{B}$ as
\begin{equation*}
 \Delta(B;\mu,\nu) := \mu(B) - \nu(B),
\end{equation*}
and the \emph{geometric $L_\infty$-discrepancy} of $\mu$ and $\nu$ with respect to $\mathcal{B}$ as
\begin{equation}\label{l_infty}
 \disc_\infty(\mathcal{B};\mu,\nu) := \sup_{B\in\mathcal{B}}|\Delta(B;\mu,\nu)|.
\end{equation}
Instead of considering the geometric $L_\infty$-discrepancy, i.e., the supremum
norm of the local discrepancy, one may prefer to consider different norms of the
local discrepancy. 
If, e.g., the class of test sets $\mathcal{B}$ is endowed with a $\sigma$-algebra
$\Sigma(\mathcal{B})$ and a probability measure $\omega$ on $\Sigma(\mathcal{B})$,
and the restrictions of $\mu$ and  $\nu$ to $\mathcal{B}$ are measurable functions, 
then one can
consider for $p\in (0,\infty)$ the \emph{geometric $L_p$-discrepancy} with 
respect to $\mathcal{B}$, defined by
\begin{equation*}
 \disc_p(\mathcal{B};\mu,\nu) 
:= \left( \int_{\mathcal{B}} |\Delta(B;\mu,\nu)|^p \, {\rm d}\omega(B) \right)^{1/p}.
\end{equation*}
In some cases other norms of the local discrepancy may be of interest, too.

In the remainder of this chapter we restrict ourselves to considering discrete measures of the form
\begin{equation}\label{nu_B}
 \nu(B) = \sum^n_{i=1} v_i 1_B(x^{(i)}),
\hspace{3ex}\text{where $v_1,\ldots, v_n\in\R$ and $x^{(1)},\ldots,x^{(n)}\in M$}.
\end{equation}
In the case where $v_i = 1/n$ for all $i$, the quality of the probability measure $\nu = \nu_X$ is 
completely determined by the quality of the ``sample points'' $x^{(1)},\ldots,x^{(n)}$.
The case where not all $v_i$ are equal to $1/n$ is of considerable interest for numerical integration
or stochastic programming (see Section~\ref{SCENARIO}). As already mentioned above, in the case of numerical integration it is natural
to relate a quadrature rule $Qf$ as in~(\ref{quadrature}) to the signed measure $\nu = \nu_Q$ in~(\ref{nu_B}). The quality of the quadrature
$Q$ is then determined by the sample points $x^{(1)}, \ldots, x^{(n)}$ and the (integration) weights $v_1, \ldots, v_n$.  

Let us provide a list of examples of specifically interesting discrepancy measures.

\begin{itemize}
\item \emph{Star discrepancy.}
Consider the situation where $M=[0,1]^d$ for some $d\in\N$,
$\Sigma$ is the usual $\sigma$-algebra of Borel sets of $M$, and
$\mu$ is the $d$-dimensional Lebesgue measure $\lambda^d$ on
$[0,1]^d$. Furthermore, let $\mathcal{C}_d$ be the class of all axis-parallel
half-open boxes anchored in zero $[0,y) = [0,y_1)\times \cdots \times [0,y_d)$, $y\in [0,1]^d$.
Then the $L_\infty$-\emph{star discrepancy} of the finite sequence $X=(x^{(i)})^n_{i=1}$ in $[0,1)^d$
is given by
\begin{equation*}
d^*_\infty(X) := \disc_\infty(\mathcal{C}_d;\lambda^{d},\nu_X) 
= \sup_{C\in\mathcal{C}_d}|\Delta(C;\lambda^d,\nu_X)|,
\end{equation*}
where 
\begin{equation}
\label{nuX}
 \nu_X(C) := \frac{1}{n} \sum^n_{i=1} 1_C(x^{(i)})
\hspace{2ex}\text{for all $C\in \Sigma$.}
\end{equation}
Thus $\nu_X$ is the counting measure 
that simply counts for given Borel sets $C$ the number of points of $X$ contained in $C$.
The star discrepancy is probably the most extensively studied discrepancy measure. Important results
about the star discrepancy and its relation to numerical integration can, e.g., be found in 
\cite{BC87, DP10, DT97, Mat99, Nie92} or the book chapters \cite{Bil12PD, CS12PD, DP12PD}.
 
Since we can identify $\mathcal{C}_d$ with $[0,1]^d$ via the mapping 
$[0,y) \mapsto y$, we may choose $\Sigma(\mathcal{C}_d)$ as the $\sigma$-algebra of Borel sets of $[0,1]^d$ and the probability measure $\omega$ on  $\Sigma(\mathcal{C}_d)$
as $\lambda^d$. Then for $1 \le p <\infty$ the \emph{$L_p$-star discrepancy}
 is given by
\begin{equation*}
d^*_p(X) :=  \disc_p(\mathcal{C}_d;\lambda^d,\nu_X) 
= \left( \int_{[0,1]^d} \left| y_1\cdots y_d - \frac{1}{n}
\sum^n_{i=1} 1_{[0,y)} (x^{(i)}) \right|^p\,{\rm dy} \right)^{1/p}.
\end{equation*}
In the last few years also norms of the local discrepancy function
different from $L_p$-norms have been studied in the literature, such as suitable 
Besov, Triebel-Lizorkin, Orlicz, and
BMO\footnote{BMO stands for ``bounded mean oscillation''.} norms, see, e.g., 
\cite{BLPV09, Hin10, Hin12, MarArXiv,Mar13, Tri10} and the book chapter \cite{Bil12PD}.

The star discrepancy is easily generalized to general measures $\nu$.
For an application that considers measures $\nu$ different from $\nu_X$ 
see Section~\ref{SCENARIO}. 

Notice that the point $0$ plays a distinguished role in the definition of the star 
discrepancy. That is why $0$ is often called the anchor of the star discrepancy.
There are discrepancy measures on $[0,1]^d$ similar to the star discrepancy
that rely also on axis-parallel boxes and on an anchor
different from $0$, such as the \emph{centered discrepancy} \cite{Hic98} or 
\emph{quadrant discrepancies} \cite{HSW04, NW09}. 
Such kind of discrepancies are, e.g., discussed in more detail in \cite{NW08,NW10}.

\item \emph{Extreme discrepancy.}
The extreme discrepancy is also known under the names \emph{unanchored discrepancy}
and \emph{discrepancy for axis-parallel boxes}. 
Its definition is analogue to the definition of the star discrepancy, except that we 
consider the class of test sets $\mathcal{R}_d$ of all half-open axis-parallel boxes
$[y,z) =[y_1,z_1)\times \cdots \times [y_d,z_d)$, $y,z \in [0,1]^d$. We may identify 
this class with the subset $\{(y,z) \,|\, y,z\in [0,1]^d, y\le z\}$ of $[0,1]^{2d}$,
and endow it with the probability measure ${\rm d}\omega(y,z) := 2^d\,{\rm d}y\,{\rm d}z$.
Thus the $L_\infty$-\emph{extreme discrepancy} of the finite sequence $X=(x^{(i)})^n_{i=1}$ in $[0,1)^d$
is given by
\begin{equation*}
d^e_\infty(X) := \disc_\infty(\mathcal{R}_d;\lambda^{d},\nu_X) 
= \sup_{R\in\mathcal{R}_d}|\Delta(R;\lambda^d,\nu_X)|,
\end{equation*}
and for $1 \le p <\infty$ the \emph{$L_p$-extreme discrepancy} is given by
\begin{equation*}
\begin{split}
d^e_p(X) :=&  \disc_p(\mathcal{R}_d;\lambda^d,\nu_X)\\ 
=& \left( \int_{[0,1]^d} \int_{[0,z)} 
\left| \prod_{i=1}^d(z_i- y_i) - \frac{1}{n}
\sum^n_{i=1} 1_{[y,z)} (x^{(i)}) \right|^p\,2^d\,{\rm d}y\,{\rm d}z \right)^{1/p}.
\end{split}
\end{equation*}
The $L_2$-extreme discrepancy was proposed as a quality measure for quasi-random 
point sets in \cite{MC94}.

To avoid confusion, it should be mentioned that the term ``extreme discrepancy'' is used by
some authors in a different way. Especially in the literature before 1980 the attribute ``extreme''
often refers to a supremum norm of a local discrepancy function , see, e.g., \cite{KN74, Whi77}.
Since the beginning of the 1990s several authors used the attribute ``extreme'' to refer to 
the set system of unanchored axis-parallel boxes, see, e.g., \cite{MC94, Nie92, NW10}.

\item \emph{$G$-discrepancy.}
The $G$- or $G$-star discrepancy is defined as the star discrepancy, except that 
the measure $\mu$ is in general not the $d$-dimensional Lebesgue
measure $\lambda^d$ on $[0,1]^d$, but some probability measure given by a distribution
function $G$ via $\mu([0,x)) = G(x)$ for all $x\in [0,1]^d$. This is
\begin{equation*}
\disc_\infty(\mathcal{C}_d;\mu,\nu_X) 
= \sup_{C\in\mathcal{C}_d} \left| G(x) - \frac{1}{n} \sum^n_{i=1} 1_{C}(x^{(i)})\right|.
\end{equation*}
The $G$-discrepancy has applications in quasi-Monte Carlo sampling, see, e.g., \cite{Okt99}. 
Further results on the $G$-star discrepancy can, e.g., be found in \cite{GR09}.

\item \emph{Isotrope discrepancy.} 
Here we have again $M=[0,1]^d$ and $\mu = \lambda^d$. As set of test sets we consider
$\II_d$, the set of all closed convex subsets of $[0,1]^d$. Then the 
\emph{isotrope discrepancy} of a set $X$ is defined as
\begin{equation}\label{isotrope}
\disc_\infty(\II_d;\lambda^{d},\nu_X) 
:= \sup_{R\in\II_d}|\Delta(R;\lambda^d,\nu_X)|.
\end{equation}
This discrepancy was proposed by Hlawka \cite{Hla64}. It has applications in probability
theory and statistics and was studied further, e.g., in \cite{Bec88, Nie72a, Sch75, Stu77, Zar70}.

\item \emph{Hickernell's modified $L_p$-discrepancy.}
For a finite point set $X\subset [0,1]^d$ and a set of variables $u\subset \{1,\ldots,d\}$ let $X_u$ denote the orthogonal projection of $X$ into the cube $[0,1]^{u}$. 
Then for $1\le p < \infty$ the
\emph{modified $L_p$-discrepancy} \cite{Hic98} of the point set 
$X$ is given by
\begin{equation}\label{hickernell_disc}
 D^*_p(X) := \left( \sum_{\emptyset \neq u \subseteq \{1,\ldots,d\}} d^*_p(X_u)^p \right)^{1/p},
\end{equation}
and for $p=\infty$ by 
\begin{equation}\label{hickernell_star_disc}
 D^*_\infty(X) := \max_{\emptyset \neq u \subseteq \{1,\ldots,d\}} d^*_\infty (X_u).
\end{equation}
In the case where $p=2$ this discrepancy was already considered by Zaremba in \cite{Zar68}.
We will discuss the calculation of the modified $L_2$-discrepancy in Section~\ref{Ltwo} and present an application of it in Section~\ref{GENOPT}.

The modified $L_p$-discrepancy is an example of a weighted discrepancy, which is the next type of discrepancy we
want to present.

\item \emph{Weighted discrepancy measures.}
In the last years \emph{weighted discrepancy measures} have become very popular,
especially in the study of tractability of multivariate and infinite-dimensional integration, see the first paper on this topic, \cite{SW98}, and, e.g.,
\cite{DP12PD, DSWW04, Gne12, HW01, LP03, NW10} and the literature mentioned therein.

To explain the idea behind the weighted discrepancy let us confine ourselves to the case where $M= [0,1]^d$ and $\nu=\nu_X$ is a discrete measure 
as in~(\ref{nuX}). (A more general definition of weighted geometric $L_2$-discrepancy, which comprises in particular
infinite-dimensional discrepancies, can be found in \cite{Gne12}.)
 We assume that there exists a one-dimensional measure $\mu^1$
 and a system $\mathcal{B}_1$ of test sets  on $[0,1]$. For $u\subseteq \{1,\ldots,d\}$ we define the product measure 
$\mu^u := \otimes_{j\in u} \mu^1$ and the system of test sets
\begin{equation*}
 \mathcal{B}_u := \bigg\{ B \subseteq [0,1]^u \,\bigg|\, B = \prod_{j\in u} B_j\,,\, B_j\in \mathcal{B}_1 \bigg\}
\end{equation*}
on $[0,1]^u$. Again we denote the projection of a set $X \subseteq [0,1]^d$ to $[0,1]^u$ by $X_u$.
Put $\mathcal{B}:= \mathcal{B}_{\{1,\ldots,d\}}$ and $\mu:= \mu^{\{1,\ldots,d\}}$. 
Let $(\gamma_u)_{u\subseteq \{1,\ldots,d\}}$ be a family of \emph{weights}, i.e., of non-negative numbers. 
Then the \emph{weighted $L_\infty$-discrepancy}
$d^*_{\infty,\gamma}(X)$ is given by
\begin{equation}\label{weighted_disc}
d^*_{\infty,\gamma}(X) := \disc_{\infty,\gamma}(\mathcal{B}; \mu, \nu_X) = \max_{\emptyset \neq u \subseteq \{1,\ldots,d\}} \gamma_u  \disc_\infty(\mathcal{B}_u; \mu^u, \nu_{X_u}).
\end{equation}
If furthermore there exists a probability measure $\omega = \omega^1$ on $\mathcal{B}_1$, put $\omega^u := \otimes_{j\in u} \omega^1$
for $u\subseteq \{1,\ldots,d\}$. The \emph{weighted $L_p$-discrepancy} 
$d^*_{p,\gamma}(X)$ is then defined by
\begin{equation*}
d^*_{p,\gamma}(X) := \disc_{p,\gamma}(\mathcal{B}; \mu, \nu_X) = \left( \sum_{\emptyset \neq u \subseteq \{1,\ldots,d\}} \gamma_u  \disc_p(\mathcal{B}_u; \mu^u, \nu_{X_u})^p \right)^{1/p},
\end{equation*}
where
\begin{equation*}
\disc_p(\mathcal{B}_u; \mu^u, \nu_{X_u})^p = \int_{\mathcal{B}_u} \left| \mu^u(B_u) - \nu_{X_u}(B_u) \right|^p \, {\rm d}\omega^u (B_u).
\end{equation*}
Hence weighted discrepancies do not only measure the uniformity of a point set $X \subset [0,1]^d$ in $[0,1]^d$, but also take
into consideration the uniformity of 
projections $X_u$ of $X$ in $[0,1]^u$. Note that Hickernell's modified $L_p$-discrepancy,
see~(\ref{hickernell_disc}), is a weighted $L_p$-star discrepancy for the particular family of weights $(\gamma_u)_{u\subseteq \{1,\ldots,d\}}$
where $\gamma_u=1$ for all $u$. 
\end{itemize} 

Other interesting discrepancy measures in Euclidean spaces as, e.g., 
discrepancies with respect to half-spaces, balls, convex polygons or rotated boxes,
can be found in \cite{BC87, Cha00, CT07, Mat99} and the literature mentioned therein.
A discrepancy measure that is defined on a flexible region, i.e., on a certain kind of parametrized variety
$M = M(m)$, $m\in (0,\infty)$, of measurable subsets of $[0,1]^d$, is the 
\emph{central composite discrepancy} proposed in \cite{CH10}.
For discrepancy measures on manifolds as, e.g., the \emph{spherical cap discrepancy},
we refer to \cite{BCCGST11, DP12PD, DT97} and the literature listed there.

There are further figures of merits known to measure the quality of points sets that are no geometric discrepancies in the sense
of our definition. Examples include the classical and the dyadic \emph{diaphony} \cite{HL97, Zin76}
or the figure of merit $R(z,n)$ \cite{JS92, Nie92, SJ94}, which are closely related to numerical integration (see also the comment at the 
beginning of Section~\ref{approximation}). 
We do not discuss such alternative figures of merit here, but focus solely on geometric discrepancy measures. In fact, we confine ourselves to
the discrepancies that can be found in the list above.
The reason for this is simple: Although deep theoretical results have been published for other geometric discrepancies, there have been,
to the best of our knowledge, no serious attempts to evaluate these geometric discrepancies efficiently.
Efficient calculation or approximation algorithms were developed almost exclusively for discrepancies that are based on axis-parallel rectangles,
such as the star, the extreme or the centered discrepancy, and weighted versions thereof.
We briefly explain in the case of the isotrope discrepancy at the beginning of Section~\ref{L_INFTY} 
the typical problem that appears if one wants to approximate other geometric discrepancies than those based on axis-parallel rectangles.

This book chapter is organized as follows:
In Section~\ref{Ltwo} we consider $L_2$-dis\-cre\-pan\-cies. In Section~\ref{WARNOCK} we explain why many of these discrepancies can be calculated exactly 
in a straightforward manner with $O(n^2\,d)$ operations, where $n$ denotes (as always) the number of points in $X$ and $d$ the dimension.
In Section~\ref{HEINRICH} we discuss some asymptotically faster algorithms which allow for an evaluation of the $L_2$-star and related
$L_2$-discrepancies in time $O(n \log n)$ (where this time the constant in the big-$O$-notation depends on $d$).
The problem of calculating $L_2$-discrepancies is the one for which the fastest algorithms are available. As we will see in the following sections,
for $p\neq 2$ there are currently no similarly efficient methods known.

In Section~\ref{L_INFTY} we discuss the calculation of the $L_\infty$-star discrepancy, which is the most prominent discrepancy measure.
To this discrepancy the largest amount of research has been devoted so far, both for theoretical and practical reasons. We remark on
known and possible generalizations to other $L_\infty$-discrepancy measures.
In Section~\ref{EAC} we present elementary algorithms 
to calculate the star discrepancy exactly. These algorithms are beneficial in low dimensions,
but clearly suffer from the curse of dimensionality. Nevertheless, the ideas used for these algorithms are fundamental for the following
Subsections of Section~\ref{L_INFTY}. In Section~\ref{DEM} we discuss the more sophisticated algorithm of Dobkin, Eppstein and Mitchell,
which clearly improves on the elementary algorithms.
In Section~\ref{Complexity} we review recent results about the complexity of exactly calculating the star discrepancy.
These findings lead us to study approximation algorithms in Section~\ref{approximation}. Here we present several 
different approaches.

In Section~\ref{L_P} we discuss the calculation of $L_p$-discrepancy measures for values of $p$ other than $2$ and $\infty$.
This Section is the shortest one in this book chapter, due to the relatively small amount of research that has been done on this topic.

In Section~\ref{APPLICATIONS} we discuss three applications of discrepancy calculation and approximation algorithms in more detail.
These applications are the quality testing of points (Section~\ref{QUALTEST}), the generation of low-discrepancy point sets via
an optimization approach (Section~\ref{GENOPT}), and scenario reduction in stochastic programming (Section~\ref{SCENARIO}).
The purpose of this section is to show the reader more recent applications and to give her a feeling of typical instance sizes that
can be handled and problems that may occur.

\section{Calculation of $L_2$-Discrepancies}
\label{Ltwo}

$L_2$-discrepancies are often used as quality measures for sets of sample points. One reason for this is 
the fact that geometric $L_2$-discrepancies
are equal to the worst-case integration error on corresponding reproducing kernel Hilbert spaces and the 
average-case integration error
on corresponding larger function spaces, see the research 
articles~\cite{FH96, Gne12, Hic98,  HW01, NW09, SW98, Woz91, Zar68} or the surveys in~\cite{DP12PD, NW10}.  

An additional advantage of the $L_2$-star discrepancy and related $L_2$-discrepancies is that they can be explicitly
computed at cost $O(dn^2)$, see Section~\ref{WARNOCK} below. Faster algorithms that are particularly
beneficial for lower dimension $d$ and larger number of points $n$ will be presented in Section~\ref{HEINRICH}.

\subsection{Warnock's Formula and Generalizations}
\label{WARNOCK}

It is easily verified by direct calculation that the $L_2$-star discrepancy of a given $n$-point set 
$X=(x^{(i)})^n_{i=1}$ in dimension $d$
can be calculated via Warnock's formula~\cite{War72} 
\begin{equation}
 \label{warnock}
d^*_2(X) = \frac{1}{3^d} 
- \frac{2^{1-d}}{n} \sum^n_{i=1} \prod_{k=1}^d (1-(x^{(i)}_k)^2)
+ \frac{1}{n^2} \sum^n_{i,j=1} \prod_{k=1}^d \min\{1-x^{(i)}_k, 1-x^{(j)}_k\}
\end{equation}
with $O(dn^2)$ arithmetic operations. As pointed out in \cite{FH96, Mat98}, the computation requires a sufficiently high precision, since the three terms in the formula are
usually of a considerably larger magnitude than the resulting $L_2$-star discrepancy.
A remedy suggested by T.~T.~Warnock \cite{War13} is to subtract off the expected value of each summand in formula (\ref{warnock}) (assuming that all coordinate
values $x^{(i)}_k$ are uniformly and independently distributed) and to add it back at the end of the computation. This means we write down (\ref{warnock}) in
the equivalent form
\begin{equation}
\begin{split}
 \label{warnock_trick}
d^*_2(X) &= \frac{1}{n} \left[ \frac{1}{2^d} - \frac{1}{3^d} \right] 
- \frac{2^{1-d}}{n} \sum^n_{i=1} \left[ \prod_{k=1}^d (1-(x^{(i)}_k)^2) - \left( \frac{2}{3} \right)^d \right]\\
&+ \frac{1}{n^2} \left( \sum^n_{\stackrel{i,j=1}{i\neq j}} \left[ \prod_{k=1}^d \min\{1-x^{(i)}_k, 1-x^{(j)}_k\} - \frac{1}{3^d} \right]
+ \sum^n_{i=1} \left[ \prod_{k=1}^d (1-x^{(i)}_k) - \frac{1}{2^d} \right] \right),
\end{split}
\end{equation}
and calculate first the terms inside the brackets $[\ldots ]$ and sum them up afterwards.
These terms are, in general, more well-behaved than the terms appearing in the original formula (\ref{warnock}), and the additional 
use of Kahan summation \cite{Kah65} helps to further reduce rounding errors \cite{War13}. 

For other $L_2$-discrepancies similar formulas can easily be deduced by direct calculation. 
So we have, e.g., that the extreme $L_2$-discrepancy of $X$ can be written as
\begin{equation}
 \label{war_extreme}
\begin{split}
d^e_2(X) = &\frac{1}{12^d} 
- \frac{2}{6^d\,n} \sum^n_{i=1} \prod_{k=1}^d (1-(x^{(i)}_k)^3 - (1-x^{(i)}_k)^3) \\
&+ \frac{1}{n^2} \sum^n_{i,j=1} \prod_{k=1}^d \min\{x^{(i)}_k, x^{(j)}_k \} \min\{1-x^{(i)}_k, 1-x^{(j)}_k\},
\end{split}
\end{equation}
cf. \cite[Sect.~4]{Hei96}, and the weighted $L_2$-star discrepancy of $X$ for \emph{product weights} $\gamma_u = \prod_{j\in u} \gamma_j$,
$\gamma_1 \ge \gamma_2 \ge \cdots \ge \gamma_d \ge 0$, as 
\begin{equation}
 \label{war_weight}
\begin{split}
d^*_{2, \gamma}(X) = &\prod_{k=1}^d \left( 1+ \frac{\gamma_k^2}{3} \right) 
- \frac{2}{n} \sum^n_{i=1} \prod_{k=1}^d \left( 1 + \gamma^2_k \frac{1 -(x^{(i)}_k)^2}{2} \right)\\
&+ \frac{1}{n^2} \sum^n_{i,j=1} \prod_{k=1}^d \left( 1 - \gamma^2_k \min\{1-x^{(i)}_k, 1-x^{(j)}_k\} \right),
\end{split}
\end{equation}
cf. also \cite[(6)]{DP12PD}.
In particular, the formula holds for the modified $L_2$-discrepancy (\ref{hickernell_disc})
that corresponds to the case where all weights $\gamma_j$, $j=1,2,\ldots$, are equal to $1$.
Notice that formulas (\ref{war_extreme}) and (\ref{war_weight}) can again be evaluated with 
$O(dn^2)$ arithmetic operations. In the case of the weighted
$L_2$-star discrepancy this is due to the simple structure of the product weights, whereas for an arbitrary family of weights
$(\gamma_u)_{u\subseteq \{1,\ldots,d\}}$ the cost of computing $d^*_{2,\gamma}(X)$ exactly will usually be of order 
$\Omega(2^d)$. 

\subsection{Asymptotically Faster Methods}
\label{HEINRICH}

For the $L_2$-star discrepancy
S.~Heinrich  \cite{Hei96} developed an 
algorithm which is asymptotically faster than the direct calculation of (\ref{warnock}). For fixed $d$ it uses at most $O(n\log^d n)$
elementary operations; here the implicit constant in the big-$O$-notation
depends on~$d$. This running time can be 
further reduced to $O(n\log^{d-1} n)$ by using a modification noted by K.~Frank and
S.~Heinrich in \cite{FH96}.

Let us start with the algorithm from \cite{Hei96}. For a quadrature rule 
\begin{equation*}
 Qf = \sum^n_{i=1} v_i f(x^{(i)}), 
\hspace{2ex}\text{with $v_i\in\R$ and $x^{(i)}\in [0,1]^d$ for all $i$},
\end{equation*}
we define the signed measure $\nu_Q$ by
\begin{equation*}
 \nu_Q(C) := Q(1_C) = \sum^n_{i=1} v_i 1_C(x^{(i)})
\hspace{2ex}\text{for arbitrary $C\in \mathcal{C}_d$.} 
\end{equation*}
Then it is straightforward to calculate 
\begin{equation}\label{heinrich}
\begin{split}
d^*_2(Q) := &\disc_2( \mathcal{C}_d; \lambda^d, \nu_Q)\\
= &\frac{1}{3^d} 
- 2^{1-d} \sum^n_{i=1} v_i \prod_{k=1}^d (1-(x^{(i)}_k)^2)
+ \sum^n_{i,j=1} v_i v_j \prod_{k=1}^d \min\{1-x^{(i)}_k, 1-x^{(j)}_k\}.
\end{split}
\end{equation}
If we are interested in evaluating this generalized version of (\ref{warnock})
in time $O(n \log^d n)$ or $O(n\log^{d-1} n)$, it obviously only remains to take
care of the efficient calculation of 
\begin{equation*}
 \sum_{i,j=1}^n v_iv_j \prod_{k=1}^d \min\{y^{(i)}_k, y^{(j)}_k\},
\hspace{2ex}\text{where $y^{(i)}_k := 1-x^{(i)}_k$ for $i=1,\ldots,n$.}
\end{equation*}
In the course of the algorithm we have actually to take care of a little bit 
more general quantities: Let $A=((v_i, y^{(i)}))^n_{i=1}$ and
$B=((w_i, z^{(i)}))^m_{i=1}$, where $n,m\in \N$, $v_i, w_i \in \R$ and $y^{(i)}, z^{(i)} 
\in [0,1]^d$ for all $i$.  
Put
\begin{equation*}
 D(A,B,d) := \sum^n_{i=1}\sum^m_{j=1} v_iw_j \prod^d_{k=1} \min\{y^{(i)}_k, z^{(j)}_k\}.
\end{equation*}
We allow also $d=0$, in which case we use the convention that the ``empty
product'' is equal to $1$. 

The algorithm is based on the following observation: If $d\ge 1$ and $y^{(i)}_d \le
z^{(j)}_d$ for all $i\in \{1,\ldots, n\}$, $j\in \{1,\ldots,m\}$, then
\begin{equation*}
 D(A,B,d) = \sum^n_{i=1}\sum^m_{j=1} (v_i \,y^{(i)}_d) w_j \prod^{d-1}_{k=1} 
\min\{y^{(i)}_k, z^{(j)}_k\} = D(\widetilde{A}, \overline{B}, d-1),
\end{equation*}
where $\widetilde{A} = ((\tilde{v}_i, \overline{y}^{(i)}))^n_{i=1}$, 
$\overline{B} = ((w_i, \overline{z}^{(i)}))^m_{i=1}$ with $\tilde{v}_i = (v_i\, y^{(i)}_d)$
and $\overline{y}^{(i)} = (y^{(i)}_k)^{d-1}_{k=1}$ and $\overline{z}^{(i)} = (z^{(i)}_k)^{d-1}_{k=1}$.
Hence we have reduced the dimension parameter $d$ by $1$. But in the case where $d=0$,
we can simply calculate 
\begin{equation}\label{base_case}
 D(A,B,0) = \left( \sum^n_{i=1} v_i \right) \left( \sum^m_{j=1} w_i \right)
\end{equation}
with cost of order $O(n+m)$. 
This observation will be exploited by the algorithm proposed by Heinrich to 
calculate $D(A,B,d)$ for given $d\ge 1$ and arrays $A$ and $B$ as above.


We describe here the version of the algorithm proposed in \cite[Sect.~2]{Hei96}; see also \cite[Sect.~5]{Mat98}. Let
$\mu$ denote the median of the $d$th components of the points $y^{(i)}$, $i=1,\ldots,n$,
from $A$. Then we split $A$ up into two smaller arrays $A_L$ and $A_R$ with $A_L$  
containing $\lfloor n/2 \rfloor$ points $y^{(i)}$ (and corresponding weights $v_i$) satisfying $y^{(i)}_d \le \mu$ and
$A_R$ containing the remaining $\lceil n/2 \rceil$ points (and corresponding weights)
satisfying $y^{(i)}_d \ge \mu$. Similarly, we split up $B$ into the two smaller arrays
$B_L$ and $B_R$ that contains the points (and corresponding weights) from $B$
whose $d$th components are less or equal than $\mu$ and greater than $\mu$, respectively.

Since we may determine $\mu$ with the help of a \emph{linear-time median-finding algorithm}
in time $O(n)$ (see, e.g., \cite[Ch.~3]{AHU74}), the whole partitioning procedure can be done in time $O(n+m)$.
With the help of this partitioning we can exploit the basic idea of the algorithm to 
obtain
\begin{equation}\label{hei-mat}
D(A,B,d) = D(A_L,B_L,d) + D(A_R, B_R,d) + D(\widetilde{A}_L, \overline{B}_R, d-1)
+ D(\overline{A}_R, \widetilde{B}_L, d-1),
\end{equation}
where, as above, $\widetilde{A}_L$ is obtained from $A_L$ by deleting the 
$d$th component of the points $y^{(i)}$ and substituting the weights $v_i$
by $v_i\, y^{(i)}_d$, and $\overline{A}_R$ is obtained from $A_R$ by also
deleting the $d$th component of the points $y^{(i)}$, but keeping the weights
$v_i$. In an analogous way we obtain $\widetilde{B}_L$ and $\overline{B}_R$, 
respectively.

The algorithm uses the step (\ref{hei-mat}) recursively in a \emph{divide-and-conquer}
manner. The ``conquer'' part consists of three base cases, which are solved directly.

The first base case is $m=0$; i.e., $B=\emptyset$. Then $D(A,B,d) = 0$.

The second one is the case $d=0$ already discussed above, where we simply use
formula (\ref{base_case}) for the direct calculation of $D(A,B,0)$ at cost at most $O(n+m)$.

The third base case is the case where $|A|=1$. Then we can compute directly
\begin{equation*}
 D(A,B,d) = v_1 \sum^m_{j=1}w_j \prod^d_{k=1} \min\{y^{(1)}_k, z^{(j)}_k\}.
\end{equation*}
This computation costs at most $O(m)$.

An inductive cost analysis reveals that the cost of this algorithm to calculate
$D(A,B,d)$ is of order $O((n+m) \log^d (n+1))$, see \cite[Prop.~1]{Hei96}.
As already said, the implicit constant in the big-$O$-notation depends on $d$.
J.~Matou\v{s}ek provided in \cite{Mat98} a running time analysis of Heinrich's 
algorithm that also takes care of its dependence on the dimension $d$ and compared it to the cost of the straightforward calculation of (\ref{heinrich}). From this analysis
one can conclude that Heinrich's algorithm reasonably outperforms the straightforward method if $n$ is larger than (roughly) $2^{2d}$; for  
details see \cite[Sect.~5]{Mat98}. 
Moreover, \cite{Hei96, Mat98} contain modifications of the algorithm and remarks on a practical implementation. 
Furthermore, Heinrich provides some numerical examples
with the number of points ranging from $1,024$ to $65,536$ in dimensions ranging from $1$ to $8$
comparing the actual computational effort of his method and the direct calculation of (\ref{heinrich}), see \cite[Sect.~5]{Hei96}.
In these examples his method was always more efficient than performing the direct  calculation; essentially, 
the advantage grows if the number of points increases, but shrinks if the dimension increases. 

As pointed out by Heinrich in \cite[Sect.~4]{Hei96}, his algorithm can be modified easily to calculate $L_2$-extreme discrepancies
instead of $L_2$-star discrepancies with essentially the same effort. Furthermore, he describes how to generalize his algorithm to
calculate ``$r$-smooth'' $L_2$-discrepancies, which were considered in \cite{Pas93, Tem90}. (Here the smoothness parameter $r$ is a non-negative integer. If 
$r=0$, we regain the $L_2$-star discrepancy. If $r>0$, then the $r$-smooth discrepancy is actually not any more a geometric $L_2$-discrepancy in the 
sense of our definition given in Section~\ref{INTRO}.)

Heinrich's algorithm can be accelerated with the help of the following observation
from \cite{FH96}: Instead of employing (\ref{base_case}) for the base case 
$D(A,B,0)$, it is possible to evaluate already the terms $D(A,B,1)$ that occur
in the course of the algorithm. If we want to calculate $D(A,B,1)$, we assume that the
elements $y^{(1)}, \ldots, y^{(n)}\in [0,1]$
from $A$ and $z^{(1)},\ldots, z^{(m)}\in [0,1]$ from $B$ are already in increasing order. 
This can be ensured by using a standard sorting algorithm to preprocess the input at cost 
$O((n+m) \log(n+m))$. Now we determine for each
$i$ an index $\nu(i)$ such that $y^{(i)} \ge z^{(j)}$ for $j=1,\ldots,\nu(i)$ and
$y^{(i)}< z^{(j)}$ for $j=\nu(i)+1,\ldots,m$. If this is done successively, starting
with $\nu(1)$, then this can be done at cost $O(n+m)$. Then
\begin{equation}\label{base_case_1}
 D(A,B,1) = \sum^n_{i=1} v_i \left( \sum^{\nu(i)}_{j=1} w_j z^{(j)} 
+ y^{(i)} \sum^m_{j=\nu(i)+1} w_j\right),
\end{equation}
and the right hand side can be computed with $O(n+m)$ operations if the inner sums are added up
successively. Thus the explicit evaluation of $D(A,B,1)$ can be done at total cost
$O(n+m)$. Using this new base case (\ref{base_case_1}) instead of (\ref{base_case})
reduces the running time of the algorithm to $O((n+m)\log^{d-1}(n+1))$ (as can easily be checked by
adapting the proof of \cite[Prop.~1]{Hei96}).

The original intention of the paper \cite{FH96} is in fact to efficiently calculate the $L_2$-star
discrepancies (\ref{heinrich}) of \emph{Smolyak quadrature rules}.  These quadrature rules are also known under different names as, e.g., \emph{sparse grid
methods} or \emph{hyperbolic cross points}, see, e.g., \cite{GG98, NR96, Smo63, Tem90, WW95, Zen91} and the literature mentioned therein.
Frank and Heinrich exploit that a $d$-dimensional Smolyak quadrature rule is uniquely determined by a sequence of one-dimensional quadratures,
and in the special case of composite quadrature rules even by a single one-dimensional quadrature. Their algorithm computes the $L_2$-star
discrepancies of Smolyak quadratures at cost $O(N\log^{2-d} N + d\log^4 N)$ for a general sequence of one-dimensional quadratures
and at cost $O(d\log^4 N)$ in the special case of composite quadrature rules; here $N$ denotes the number of quadrature points
used by the $d$-dimensional Smolyak quadrature. This time the implicit constants in the big-$O$-notation do not depend on the dimension $d$.
With the help of their algorithm Frank and Heinrich are able to calculate the $L_2$-star discrepancy for extremely large numbers of integration
points as, e.g., roughly $10^{35}$ points in dimension $d=15$. For the detailed description of the algorithm and numerical experiments we refer 
to \cite{FH96}.

Notice that both algorithms from \cite{FH96, Hei96} use as a starting point formula (\ref{heinrich}). Since the three summands appearing in (\ref{heinrich}) are of similar size, 
the algorithms should be executed with a sufficiently high precision to avoid cancellation effects.

\subsubsection*{Notes}
Related to the problem of calculating $L_2$-discrepancies of given point sets is the problem 
of computing the smallest possible $L_2$-discrepancy of all point sets of a given size $n$.
For the $L_2$-star discrepancy and arbitrary dimension 
$d$ the smallest possible discrepancy value of all point sets of size $n$ was derived in
\cite{PVC06} for $n=1$ and in \cite{LP07} for $n=2$.  

Regarding the $L_2$-star discrepancy, one should mention that this discrepancy can be 
a misleading measure of uniformity for sample sizes $n$ smaller than $2^d$. For instance, Matou\v{s}ek pointed out
that for small $n$ the pathological point set that consists of $n$ copies of the point
$(1,\ldots,1)$ in $[0,1]^d$ has almost the best possible $L_2$-star discrepancy,
see \cite[Sect.~2]{Mat98}. A possible remedy is to consider a weighted version of the
$L_2$-star discrepancy instead, as, e.g., the modified $L_2$-discrepancy.  

Matou\v{s}ek's observation may also be interpreted in the context of numerical integration. The $L_2$-star discrepancy
is equal to the worst-case error of quasi-Monte Carlo (QMC) integration on the unanchored Sobolev space. More precisely,
we have for a finite sequence $X=(x^{(i)})^n_{i=1}$ in $[0,1]^d$ that
\begin{equation*}
 d^*_2(X) = \sup_{f \in B} \left| \int_{[0,1]^d} f(x) \,{\rm d}x - Q(f) \right|,
\end{equation*}
where $B$ is the norm unit ball of the unanchored Sobolev space and $Q$ is the QMC algorithm
\begin{equation*}
 Q(f) = \frac{1}{n} \sum^n_{i=1} f(x^{(i)});
\end{equation*}
see, e.g., \cite[Chapter 9]{NW10}.
Now Matou\v{s}ek's observation indicates that if for given $n$ smaller than $2^d$ one is interested in minimizing the worst-case 
integration error with the help
of a general quadrature rule of the form (\ref{quadrature}),
then one should not use QMC algorithms with equal integration weights $1/n$.
In fact, already normalized QMC algorithms with suitably chosen equal integration weights $a = a(n,d) < 1/n$ as stated in 
\cite[(10.12)]{NW10} improve over conventional QMC algorithms
with weights $1/n$; for a detailed discussion see \cite[Sect.~10.7.6]{NW10}. This suggests that for $n$ smaller than $2^d$ the 
$L_2$-star discrepancy modified by substituting the factor $1/n$ by $a < 1/n$ from \cite[(10.12)]{NW10} may be a better measure of uniformity
than the $L_2$-star discrepancy itself. 

\section{Calculation of $L_\infty$-Discrepancies}
\label{L_INFTY}

In this section we survey algorithms which can be used to calculate or approximate the $L_\infty$-star discrepancy. 
Most of these algorithms have a straightforward extension to other ``$L_\infty$-rectangle discrepancies'', as, e.g., to the
extreme discrepancy discussed above, the centered 
discrepancy \cite{Hic98}, or other quadrant discrepancies \cite{HSW04, NW10}.
Algorithms for the $L_\infty$-star discrepancy are also necessary to compute or estimate weighted $L_\infty$-star discrepancies. Let us, e.g., assume that we are interested in 
finding tight upper or lower bounds for the weighted $L_\infty$-discrepancy $d^*_{\infty, \gamma}(X)$, as defined in (\ref{weighted_disc}).
Then we may divide the family of weights $(\gamma_u)_{u\subseteq\{1,\ldots,d\}}$ into a set $S$ of suitably small weights and  a set $L$ of larger weights  and
use the fact that the star discrepancy has the following monotonicity behavior with respect to the dimension: If $u\subseteq v$, then 
$d^*_\infty(X_u) \le d^*_\infty(X_v)$. 
We can use the algorithms discussed below to calculate or bound the discrepancies $d^*_{\infty}(X_u)$, $u\in L$. The remaining
discrepancies $d^*_{\infty}(X_v)$, $v\in S$, corresponding to the less important weights can be upper-bounded simply by $1$ and lower-bounded by
$\max_{u\in L\,;\, u\subset v} d^*_{\infty}(X_u)$ 
(or even by $0$ if the weights are negligible small).

In general, it is not easy to calculate $L_\infty$-discrepancies as defined in (\ref{l_infty}); the cardinality of the system $\mathcal{B}$ of test sets 
is typically infinite. Since we obviously cannot compute the local discrepancies for an infinite number of test boxes, we usually have to find a finite subset 
$\mathcal{B}_{\delta} \subset \mathcal{B}$ such that $\disc_\infty(\mathcal{B}; \mu, \nu) = \disc_\infty(\mathcal{B}_{\delta}; \mu, \nu)$ or at least
$\disc_\infty(\mathcal{B}; \mu, \nu) \le \disc_\infty(\mathcal{B}_{\delta}; \mu, \nu) + \delta$ for sufficiently small $\delta$. (This ``discretization method''
is also important for finding upper bounds for the best possible discrepancy behavior with the help of \emph{probabilistic proofs}, see, e.g., \cite{Bec84} or 
the book chapter \cite{Gne12b}.)
For most systems  $\mathcal{B}$ of test sets  this is not a trivial task. If one is, for instance, interested in the isotrope discrepancy, see (\ref{isotrope}),
it is not completely 
obvious to see how the system $\II_d$ can be substituted by a finite set system that leads to a (arbitrarily) close approximation of
$\disc_\infty(\II_d; \lambda^d, \nu_X)$. In \cite{Nie72a} H.~Niederreiter pointed out that it is sufficient to consider the smaller system of test sets
$\mathcal{E}_d$ of all open and closed polytopes $P$ contained in $[0,1]^d$ with the property that each face of $P$ is lying entirely on the boundary
of $[0,1]^d$ or contains a point of $X$. Note that $\mathcal{E}_d$ still consists of infinitely many test sets and that further work has to be done
before this observation can be used for a concrete algorithm to approximate the isotrope discrepancy.

For the star discrepancy it is easier to find useful discretizations, as we will show below.


\subsection{Calculating the Star Discrepancy in Low Dimension}
\label{EAC}

Let us have a 
closer look at the problem of calculating the $L_\infty$-star discrepancy: 
Let $X=(x^{(i)})^n_{i=1}$ be some fixed finite sequence in $[0,1)^d$. 
For convenience we introduce for an arbitrary point $y\in [0,1]^d$ the short-hands 
\begin{equation*}
 V_y := \prod^d_{i=1} y_i,
\end{equation*}
and 
\begin{equation*}
 A(y, X) : = \sum^n_{i=1} 1_{[0,y)}(x^{(i)}),
\hspace{3ex}\text{as well as}\hspace{3ex}
\overline{A}(y, X) : = \sum^n_{i=1} 1_{[0,y]}(x^{(i)}),
\end{equation*}
i.e., $V_y$ is the volume of the test box $[0,y)$, $A(y,X)$ the number
of points of $X$ lying inside the half-open box $[0,y)$, and
$\overline{A}(y,X)$ the number of points of $X$ lying in the
closed box $[0,y]$. Let us furthermore set 
\begin{equation*}
\delta(y,X) := V_y - \frac{1}{n}A(y,X) 
\hspace{3ex}\text{and}\hspace{3ex}
\overline{\delta}(y,X) := \frac{1}{n} \overline{A}(y,X) -V_y.
\end{equation*}
Putting $\delta^*(y,X) := \max\{ \delta(y,X), \overline{\delta}(y,X)\}$,
we have
\begin{equation*}
 d^*_\infty(X) = \sup_{y\in [0,1]^d} \delta^*(y,X).
\end{equation*}
We define for $j\in\{1,\ldots,d\}$
\begin{equation*}
\Gamma_j(X) := \{x^{(i)}_j \,| \, i \in \{1,...,n \} \}
\hspace{2ex}\text{and}\hspace{2ex}
\overline{\Gamma}_j(X) := \Gamma_j(X) \cup \{1\},
\end{equation*}
and put
\begin{equation*}
\Gamma(X) := \Gamma_1(X) \times \cdots \times \Gamma_d(X)
\hspace{2ex}\text{and}\hspace{2ex}
\overline{\Gamma}(X) := \overline{\Gamma}_1(X) \times \cdots \times \overline{\Gamma}_d(X).
\end{equation*}
We refer to $\Gamma(X)$ and to $\overline{\Gamma}(X)$ as \emph{grids induced by} $X$.

\begin{lemma}\label{Disfor}
Let $X=(x^{(i)})^n_{i=1}$ be a sequence in $[0,1)^d$. Then
\begin{equation}
\label{disfor}
d^*_\infty(X) = \max \ \left\{ \ \max_{y \in \overline{\Gamma}(X)} 
\delta(y,X)
\ ,\ \max_{y \in \Gamma(X)} \overline{\delta}(y,X) 
\right\}.
\end{equation}
\end{lemma}

Formulas similar to (\ref{disfor}) can be found in several places
in the literature---the first reference we are aware of is 
\cite[Thm.~2]{Nie72a}.

\begin{proof}
Consider an
arbitrary test box $[0,y)$, $y\in [0,1]^d$, see Figure~\ref{Fig_1}.
\begin{figure}
	\begin{center}
	\begin{tikzpicture}
		\def\s{6}; 
		\draw (0,0) -- (\s,0) -- (\s,\s) -- (0,\s) -- (0,0);
		\draw[style=dashed] (0.755*\s,0) -- (0.755*\s,0.776*\s) -- (0,0.776*\s);
		\draw (0.163*\s,0) -- (0.163*\s,\s);
		\draw (0.296*\s,0) -- (0.296*\s,\s);
		\draw (0.490*\s,0) -- (0.490*\s,\s);
		\draw (0.714*\s,0) -- (0.714*\s,\s);
		\draw (0.878*\s,0) -- (0.878*\s,\s);
		\draw (0,0.163*\s) -- (\s,0.163*\s);
		\draw (0,0.398*\s) -- (\s,0.398*\s);
		\draw (0,0.622*\s) -- (\s,0.622*\s);
		\draw (0,0.816*\s) -- (\s,0.816*\s);
		\draw (0,0.955*\s) -- (\s,0.955*\s);
		\draw[fill=black] (0.163*\s,0.816*\s) circle (0.05);
		\draw[fill=black] (0.296*\s,0.163*\s) circle (0.05);
		\draw[fill=black] (0.490*\s,0.398*\s) circle (0.05);
		\draw[fill=black] (0.714*\s,0.955*\s) circle (0.05);
		\draw[fill=black] (0.878*\s,0.622*\s) circle (0.05);
		\draw (0.714*\s,0.622*\s) circle (0.05);
		\draw (0.878*\s,0.816*\s) circle (0.05);
		\draw (-0.2,-0.2) node {0};
		\draw (\s,-0.2) node {1};
		\draw (-0.2,\s) node {1};
		\draw (0.755*\s+0.2,0.776*\s) node {$y$};
		\draw (0.714*\s-0.2,0.622*\s-0.2) node {$x$};
		\draw (0.878*\s+0.2,0.816*\s+0.2) node {$z$};
		\draw (0.163*\s+0.4,0.816*\s+0.3) node {$x^{(1)}$};
		\draw (0.296*\s+0.4,0.163*\s+0.3) node {$x^{(2)}$};
		\draw (0.490*\s+0.4,0.398*\s+0.3) node {$x^{(3)}$};
		\draw (0.714*\s-0.4,0.955*\s-0.3) node {$x^{(4)}$};
		\draw (0.878*\s+0.4,0.622*\s+0.3) node {$x^{(5)}$};
	\end{tikzpicture}
	\end{center}
\caption{\label{Fig_1}Some set $X=(x^{(i)})^5_{i=1}$ in $[0,1)^2$, a test
box $[0,y)$, and $x\in \Gamma(X)$, $z\in \overline{\Gamma}(X)$ with $x\le y\le z$.}
\end{figure}
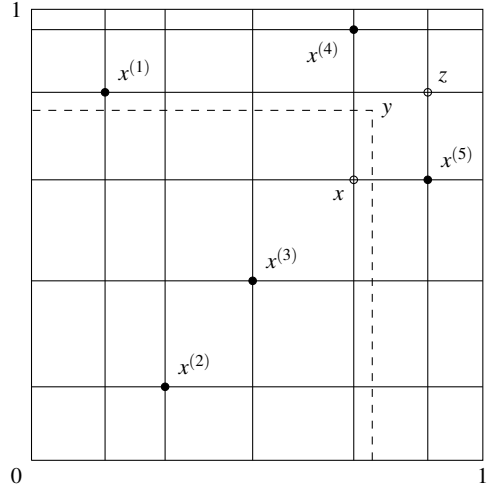

Then for every 
$j\in \{1,\ldots,d\}$
we find a maximal $x_j\in\Gamma_j(X)\cup \{0\}$ and a minimal
$z_j\in\overline{\Gamma}(X)$ satisfying $x_j\le y_j\le z_j$. 
Put $x=(x_1,\ldots,x_d)$ and $z=(z_1,\ldots,z_d)$.
We get the inequalities
\begin{equation*}
V_y -\frac{1}{n} A(y,X) = V_y -\frac{1}{n} A(z,X) 
\le V_z -\frac{1}{n} A(z,X),
\end{equation*}
and 
\begin{equation*}
\frac{1}{n} A(y,X) - V_y = \frac{1}{n}\overline{A}(x,X) - V_y 
\le \frac{1}{n}\overline{A}(x,X) - V_x.
\end{equation*}
Observing that $\overline{A}(x,X) = 0$ if $x_j=0\notin \Gamma_j(X)$ for some 
$j\in \{1,\ldots,d\}$, we see that the right hand side of (\ref{disfor}) is at least as large as $d^*_\infty(X)$. 

Let us now show that it cannot be larger than $d^*_\infty(X)$. So let
$y\in\Gamma(X)$ be given. Then we may consider for a small 
$\varepsilon>0$ the vector $y(\varepsilon)$, defined
by $y(\varepsilon)_j = \min\{y_j+\varepsilon, 1\}$ for 
$j=1,\ldots,d$. Clearly, $y(\varepsilon)\in [0,1]^d$ and
\begin{equation*}
\lim_{\varepsilon\to 0} \left( \frac{1}{n}
 A(y(\varepsilon),X) - V_{y(\varepsilon)} \right)
= \frac{1}{n} \overline{A}(y,X) - V_y.
\end{equation*} 
These arguments show that (\ref{disfor}) is valid.
\qed
\end{proof}

Lemma~\ref{Disfor} shows that an enumeration algorithm would provide us with the exact 
value of $d^*_\infty(X)$. But since the cardinality of $\Gamma(X)$ for
almost all $X$ is $n^d$, such an algorithm would
be infeasible for large values of $n$ and $d$. 
Indeed, for a random $n$-point set $X$ we have
almost surely $|\Gamma(X)|=n^d$, resulting in $\Omega(n^d)$
test boxes that we have to take into account to calculate (\ref{disfor}).
This underlines that $(\ref{disfor})$ is in general impractical if 
$n$ and $d$ are large. There are some more sophisticated methods
known to calculate the star discrepancy, which are especially 
helpful in low dimensions. In the one-dimensional case H.~Niederreiter derived
the following formula, see \cite{Nie72a} or \cite{Nie72b}.

\begin{theorem}[{\cite[Thm.~1]{Nie72a}}]
Let $X=(x^{(i)})^n_{i=1}$ be a sequence in $[0,1)$. If
$x^{(1)}\le x^{(2)} \le \cdots \le x^{(n)}$, then 
\begin{equation}
\label{nie}
d^*_\infty(X) = \max^n_{i=1} \max \left\{
\frac{i}{n}-x^{(i)} \,,\, x^{(i)} - \frac{i-1}{n} \right\} 
= \frac{1}{2n} + \max^n_{i=1} \left| x^{(i)} - \frac{2i-1}{2n} \right|.
\end{equation}
\end{theorem}

\begin{proof}
The first identity follows directly from (\ref{disfor}), since for 
$y=1\in \overline{\Gamma}(X)$ we have $V_y - \frac{1}{n} A(y,X) =0$.
The second identity follows from
\begin{equation*}
 \max \left\{ \frac{i}{n} -x^{(i)}\,,\,
x^{(i)}-\frac{i-1}{n} \right\} 
= \left| x^{(i)} - \frac{2i-1}{2n} \right| + \frac{1}{2n},
\hspace{3ex}i=1,\ldots,n.
\end{equation*}
\qed
\end{proof}

Notice that (\ref{nie}) implies immediately that for $d=1$ the set 
$\{\frac{1}{2n}, \frac{3}{2n}, \ldots, \frac{2n-1}{2n} \}$ is the uniquely determined
$n$-point set that achieves the minimal star discrepancy $1/2n$.

In higher dimension the calculation of the star discrepancy unfortunately
becomes more complicated.

In dimension $d=2$ a reduction of the number of steps to calculate
(\ref{disfor}) was achieved by L.~De~Clerk \cite{DeC86}. In 
\cite{BZ93} her formula was slightly extended and simplified by
P.~Bundschuh and Y.~Zhu. 

\begin{theorem}[{\cite[Sect.~II]{DeC86}},{\cite[Thm. 1]{BZ93}}]
Let $X=(x^{(i)})^n_{i=1}$ be a sequence in $[0,1)^2$.
Assume that $x^{(1)}_1\le x^{(2)}_1 \le \cdots \le x^{(n)}_1$ and rearrange
for each $i\in \{1,\ldots,n\}$ the numbers
$0, x^{(1)}_2,\ldots, x^{(i)}_2,1$ in increasing order and rewrite them as
$0=\xi_{i,0} \le \xi_{i,1} \le \cdots \le \xi_{i,i} \le \xi_{i,i+1}
=1$. Then
\begin{equation}
 \label{dec}
d^*_\infty(X) = \max^n_{i=0} \max^i_{k=0} \max
\left\{  \frac{k}{n} -x^{(i)}_1\xi_{i,k} \,,\,
x^{(i+1)}_1\xi_{i,k+1} - \frac{k}{n} \right\}.
\end{equation}
\end{theorem}

The derivation of this formula is mainly based on the observation that in dimension
$d\ge 2$ the discrepancy functions $\delta(\cdot,X)$ and $\overline{\delta}(\cdot,X)$ can attain
their maxima only in some of the points  $y\in \overline{\Gamma}(X)$ and 
$y\in \Gamma(X)$, respectively, which we shall call \emph{critical points}. 
The formal definition we state here is
equivalent to the one given in \cite[Sect.~4.1]{GWW12}.

\begin{definition}
\label{def:critical}
 Let $y\in [0,1]^d$. The point $y$ and the test box $[0,y)$ are
\emph{$\delta(X)$-critical}, if we have for all $0 \neq \varepsilon\in 
[0,1-y_1]\times \cdots \times [0,1-y_d]$ that $A(y+\varepsilon, X) > A(y,X)$.
The point $y$ and the test box $[0,y]$ are
\emph{$\overline{\delta}(X)$-critical}, if we have for all $\varepsilon\in [0,y)\setminus
\{0\}$ that $\overline{A}(y-\varepsilon, X) < \overline{A}(y,X)$.
We denote the set of $\delta(X)$-critical points by $\mathcal{C}(X)$ and the set of
$\overline{\delta}(X)$-critical points by $\overline{\mathcal{C}}(X)$, and we put
$\mathcal{C}^*(X) := \mathcal{C}(X) \cup \overline{\mathcal{C}}(X)$.
\end{definition}
 
In Figure~\ref{Fig_1} we have, e.g., that $y=(x^{(5)}_1, x^{(4)}_2)$ and 
$y'=(x^{(5)}_1, 1)$ are $\delta(X)$-critical points, while
$y''=(x^{(5)}_1, x^{(3)}_2)$ is not.
Furthermore, $\overline{y}=(x^{(3)}_1, x^{(3)}_2)$ and 
$\overline{y}'=(x^{(5)}_1, x^{(1)}_2)$ are $\overline{\delta}(X)$-critical points,
while $\overline{y}''=(x^{(4)}_1, x^{(2)}_2)$ is not.
This shows that, in contrast to the one-dimensional situation, for $d\ge 2$ not all
points in $\Gamma(X)$ and $\overline{\Gamma}(X)$ are critical points.

With a similar argument as in the proof of Lemma~\ref{Disfor}, the following
lemma can be established.

\begin{lemma}
Let $X = (x^{(i)})_{i=1}^n$ be a sequence in $[0,1]^d$. Then $\mathcal{C}(X) \subseteq \overline{\Gamma}(X)$ and 
$\overline{\mathcal{C}}(X) \subseteq \Gamma(X)$, as well as 
\begin{equation}
 \label{discrit1}
d^*_\infty(X) = \max
\left\{ \max_{y\in \mathcal{C}(X)} \delta(y,X)\,,\, \max_{y\in\overline{\mathcal{C}}(X)} 
\overline{\delta}(y,X) \right\}.
\end{equation}
\end{lemma}

For a rigorous proof see \cite[Sect.~II]{DeC86} or \cite[Lemma~4.1]{GWW12}.

The set of critical test boxes may be subdivided further. 
For $j=0,1,\ldots,n$ put 
\begin{equation*}
 \mathcal{C}^k(X) := \left\{ y\in \mathcal{C}(X) \,|\, A(y,X) = k \right\},
\end{equation*}
and 
\begin{equation*}
 \overline{\mathcal{C}}^k(X) := \left\{ y\in \overline{\mathcal{C}}(X) \,|\, 
\overline{A}(y,X) = k \right\}.
\end{equation*}
Then
\begin{equation}
 \label{discrit2}
d^*_\infty(X) = \max_{k=1}^n \max 
\left\{ \max_{y\in \mathcal{C}^{k-1}(X)} \delta(y,X)\,,\, \max_{y\in\overline{\mathcal{C}}^{k}(X)} 
\overline{\delta}(y,X) \right\},
\end{equation}
see \cite[Sect.~II]{DeC86}. For $d=2$ De~Clerk characterized the points in
$\mathcal{C}^k(X)$ and $\overline{\mathcal{C}}^k(X)$ and derived (\ref{dec})
under the assumption that $|\Gamma_j(X)|=n$ for $j=1,\ldots,d$ \cite[Sect.~II]{DeC86}.
Bundschuh and Zhu got rid of this assumption, which resulted in a technically more involved
proof \cite[Sect.~2]{BZ93}. 

De~Clerck used her formula (\ref{dec}) to provide explicit formulas for the discrepancies
of \emph{Hammersley point sets} \cite{Ham60} for arbitrary basis $b$ and $n=b^m$, $m\in \N$, see \cite[Sect.~III]{DeC86}. Her results generalize the formulas for the star discrepancy of two-dimensional Hammersley sets in basis $2$ provided by Halton and Zaremba \cite{HZ69}.
To establish the explicit formula, she used a recursion property of two-dimensional Hammersley point
sets $X$ of size $b^m$ \cite{Fau81} and the facts that these sets are symmetric with respect to the main diagonal of the unit square \cite{Pea82} and that their discrepancy function $\delta(\cdot,X)$ is never positive \cite{Gab63}, which implies that
for the calculation of the discrepancy of these sets one only has to consider $\overline{\delta}(X)$-critical test boxes.

In dimension $d=3$ Bundschuh and Zhu provided a formula similar to (\ref{dec}). 

\begin{theorem}[{\cite[Thm.~2]{BZ93}}]
Let $X=(x^{(i)})^n_{i=1}$ be a sequence in $[0,1)^3$.
Put $x^{(0)} := (0,0,0)$ and $x^{(n+1)}:= (1,1,1)$, 
and assume that $x^{(1)}_1\le x^{(2)}_1 \le \cdots \le x^{(n)}_1$. For $i\in \{1,\ldots,n\}$  rearrange
the second components
$x^{(0)}_2, x^{(1)}_2,\ldots, x^{(i)}_2, x^{(n+1)}_2$ in increasing order and rewrite them as
$0=\xi_{i,0} \le \xi_{i,1} \le \cdots \le \xi_{i,i} \le \xi_{i,i+1}
=1$ and denote the corresponding third components $x^{(i)}_3$, $i=0,1,\ldots, i, n+1$
by $\tilde{\xi}_{i,0}, \tilde{\xi}_{i,1},\ldots, \tilde{\xi}_{i,i+1}$. Now for fixed
$i$ and $k =0,1,\ldots,i$ rearrange $\tilde{\xi}_{i,0}, \tilde{\xi}_{i,1},\ldots, \tilde{\xi}_{i,k},\tilde{\xi}_{i,i+1}$ and rewrite them as
$0 = \eta_{i,k,0} \le \eta_{i,k,1} \le \cdots \le \eta_{i,k,k} \le \eta_{i,k,k+1} =1$.
Then 
\begin{equation}
 \label{buz}
d^*_\infty(X) = \max^n_{i=0} \max^i_{k=0} \max^k_{\ell=0} \max
\left\{  \frac{k}{n} -x^{(i)}_1\xi_{i,k}\eta_{i,k,\ell} \,,\,
x^{(i+1)}_1\xi_{i,k+1}\eta_{i,k,\ell+1} - \frac{k}{n} \right\}.
\end{equation}
\end{theorem}

The method can be generalized to arbitrary 
dimension $d$ and requires for generic
point sets roughly $O(n^d/d!)$ elementary operations.
This method was, e.g., used in \cite{WF97} to calculate the 
exact discrepancy of particular point sets, so-called rank-$1$ 
lattice rules (cf. \cite{DP10, Nie92, SJ94}), up to size $n=236$ in dimension
$d=5$ and to $n=92$ in dimension $d=6$.
But as pointed out by P.~Winker and K.-T.~Fang, for this method instances 
like, e.g., sets of 
size $n \ge 2000$ in $d=6$ are infeasible. This method can thus only be used
in a very limited number of dimensions. 

A method that calculates the exact star discrepancy of a point set in a running time
with a more favorable dependence on the dimension $d$, namely time $O(n^{1+d/2})$,
was proposed by D.~P.~Dobkin, D.~Eppstein, and D.~P.~Mitchell in \cite{DEM96}. 
We discuss this more elaborate algorithm in the next subsection. 

\subsection{The Algorithm of Dobkin, Eppstein, and Mitchell}
\label{DEM}

In order to describe the algorithm of Dobkin, Eppstein, and Mitchell~\cite{DEM96}, we
begin with a problem from computational geometry. In \emph{Klee's measure problem}, the
input is a set of $n$ axis-parallel rectangles in $\R^d$, and the problem is to compute
the volume of the union of the rectangles. The question of the best possible running time
started with V.~Klee~\cite{Klee77}, who gave an $O(n\log n)$-time algorithm for the
1-dimensional case, and asked whether this was optimal; it was later found that this is
the case~\cite{FW78}. The general case was considered by J.~L.~Bentley~\cite{Bentley77}, who 
gave an $O(n^{d-1}\log n)$-time algorithm for $d \geq 2$, thus (surprisingly) giving an
algorithm for the $d=2$ case with an asymptotic running time matching that of the $d=1$
case. 
By the lower bound for $d=1$, Bentley's algorithm is tight for $d=2$, but as we shall see,
not for $d \geq 3$. 

The essential breakthrough, which also lies behind the result of~\cite{DEM96}, was given
by Overmars and Yap~\cite{OY91}. They showed that~$\R^d$ can be partitioned (in an
input-dependent way) into~$O(n^{d/2})$ regions, where each region is an axis-parallel box,
such that the intersection of the region with the rectangles 
of the input behaves in a
particular regular way. 
Let us fix some terminology. Let $C=[a,b]$ be a region in the decomposition. 
A \emph{slab} in $C$ is an axis-parallel box contained in $C$ which has full length in all
but at most one dimension, i.e., a box $\prod_{j=1}^d I_j$ where $I_j=[a_j,b_j]$ for all
but at most one~$j$.
A finite union of slabs is called a \emph{trellis}. Overmars and Yap show the following
result. 

\begin{theorem}[\cite{OY91}]
Let~$d$ be a fixed dimension. Given a set of~$n$ rectangles in~$\R^d$, there is a
partitioning of the space into~$O(n^{d/2})$ regions where for every region, the
intersection of the region with the union of all rectangles of the input forms a trellis.
\end{theorem}

An algorithm to enumerate this partition, and a polynomial-time algorithm
for computing the volume of a trellis, would now combine into an~$O(n^{d/2+c})$-time
algorithm for Klee's measure problem, for some constant~$c>0$; Overmars and Yap further
improve it to $O(n^{d/2}\log n)$ time (and~$O(n)$ space) by using dynamic data structures and
careful partial evaluations of the decomposition. Recently, T.~M.~Chan~\cite{Chan10} gave a
slightly improved running time of $n^{d/2}2^{O(\log^* n)}$, where~$\log^* n$ denotes the
iterated logarithm of~$n$, i.e., the number of times the logarithm function must be
iteratively applied before the result is less than or equal to 1. 

While no direct reductions between the measure problem and discrepancy computation are
known, the principle behind the above decomposition is still useful. 
Dobkin et al.~\cite{DEM96} apply it via a dualization, turning each point $x$ into an
orthant $(x,\infty)$,
and each box $B=[0,y)$ into a point $y$,
so that $x$ is contained
in $B$ if and only if the point $y$ is contained in the orthant. The problem of star
discrepancy computation can then be solved by finding, for each $i \in \{1,\ldots,n\}$,
the ``largest'' and ``smallest'' point $y \in \overline{\Gamma}(X)$ contained in at most respectively at least $i$
orthants; here largest and smallest refer to the value~$V_y$ defined previously, 
i.e., the volume of the box $[0,y)$. Note that this is the dual problem to (\ref{discrit2}).
As Dobkin et al.~\cite{DEM96} show, the partitioning of~$\R^d$
of Overmars and Yap can be used for this purpose. In particular, the base case of a region
$B$ such that the intersections of the input rectangles with $B$ form a trellis, 
corresponds
for us to a case where for every point $x$ there is at most one coordinate $j \in \{1,\ldots,d\}$
such that for any $y \in B$ only the value of $y_j$ determines whether $x \in [0,y)$.
Given such a base case, it is not difficult to compute the maximum discrepancy relative to
points~$y\in B$ in polynomial time. 

We now sketch the algorithm in some more detail, circumventing the dualization step for a
more direct presentation. For simplicity of presentation, we are assuming that for each
fixed $j \in \{1,\ldots,d\}$, the coordinates $x_j^{(1)}, \ldots, x_j^{(n)}$ are pairwise
different. We will also ignore slight issues with points that lie exactly on a
region boundary. 

We will subdivide the space $[0,1]^d$ recursively into regions of the form $[a_1,b_1]
\times \ldots \times [a_i,b_i] \times [0,1]^{d-i}$; call this a \emph{region at level
  $i$}. We identify a region with $[a,b]$; if the region is at level $i<d$, then for $j>i$
we have $a_j=0$ and $b_j=1$. For $j \in \{1,\ldots,i\}$, we say that a point $x$ is
\emph{internal in dimension $j$}, relative to a region $[a,b]$ at level $i$, if $x$ is
contained in $[0,b)$ and $a_j < x_j < b_j$. 
We will maintain two invariants as follows.
\begin{enumerate}
\item For every point $x$ contained in the box $[0,b)$, there is at most one coordinate 
$j \in \{1,\ldots,i\}$ such that $x$ is internal in dimension $j$. 
\item For any region at level $i>0$ and any $j \in \{1,\ldots,i\}$ there are only
$O(\sqrt n)$ points internal in dimension $j$. 
\end{enumerate}
Once we reach a region at level $d$, called a \emph{cell}, the first condition ensures 
that we reach the above-described base case, i.e., that every point contained in
the box $[0,b)$ is internal in at most one dimension.
The second condition, as we will see, ensures that the decomposition can be performed
while creating only $O(n^{d/2})$ cells.


The process goes as follows. We will describe a recursive procedure, subdividing each
region at level $i<d$ into $O(\sqrt{n})$ regions at level $i+1$, ensuring
$O(\sqrt{n}^d)=O(n^{d/2})$ cells in the final decomposition. 
In the process, if $[a,b]$ is the region currently being subdivided, we let $X_b$ be the set 
of points contained in the box $[0,b)$, and $X_I$ the subset of those points which are
internal in some dimension $j\leq i$. 
We initialize the process with the region $[0,1]^d$ at level 0, with $X_b$ being the full
set of points of the input, and $X_I=\emptyset$. 

Given a region $[a,b]$ at level $i<d$, with $X_b$ and $X_I$ as described, we will 
partition along dimension $i+1$ into segments $[\xi_j,\xi_{j+1}]$, for some $\xi_j$,
$j\in \{1, \ldots, \ell\}$, where~$\ell$ is the number of subdivisions. 
Concretely, for each $j \in \{1,\ldots,\ell-1\}$
we generate a region $[a^{(j)},b^{(j)}]=[a_1,b_1] \times \ldots \times [a_i,b_i] \times
[\xi_j,\xi_{j+1}] \times [0,1]^{d-i-1}$ at level $i+1$, set $X_b^{(j)}=X_b \cap [0,b^{(j)})$
and $X_I^{(j)} = (X_I \cap [0,b^{(j)})) \cup \{x \in X_b^{(j)}: \xi_{j} < x_{i+1} < \xi_{j+1}\}$,
and recursively process this region, with point set $X_b^{(j)}$ and internal points $X_I^{(j)}$. 
Observe that the points added to $X_I^{(j)}$ indeed are internal points, in dimension $i+1$.
The coordinates $\xi_j$ are chosen to fulfill the following conditions.
\begin{eqnarray}
\text{For all}&&\!\!\!\! x \in X_I, \text{ we have } x_{i+1} \in \{\xi_1,\ldots, \xi_\ell\}. \label{eqn:prop1}
\\ 
\text{For all}&&\!\!\!\! j \in \{1,\ldots,\ell-1\}, \text{ we have } |\{x \in X_b: \xi_j < x_{i+1} < \xi_{j+1}\}|
= O(\sqrt{n}).\,
\label{eqn:prop2}
\end{eqnarray}
We briefly argue that this is necessary and sufficient to maintain our two invariants
while ensuring that~$\ell=O(\sqrt{n})$ at every level. Indeed, if condition (\ref{eqn:prop1})
is not fulfilled, then the point $x$ is internal in two dimensions in some region 
$[a^{(j)},b^{(j)}]$, 
and if condition (\ref{eqn:prop2}) is not fulfilled for some $j$,
then $|X_I^{(j)}|$ is larger than
$O(\sqrt{n})$ in the same region. 
To create a set of coordinates $\xi_j$ fulfilling these conditions is not difficult;
simply begin with the coordinate set $\{x_{i+1}: x \in X_I\}$ to
fulfill~(\ref{eqn:prop1}), and insert additional 
coordinates as needed to satisfy~(\ref{eqn:prop2}). This requires at most
$|X_I|+(n/\sqrt{n})$ coordinates; thus $\ell=O(\sqrt n)$ as required, and one finds
inductively that both invariants are maintained at every region created at level $i+1$. 


To finally sketch the procedure for computing the discrepancy inside a given cell,
let~$B=[a,b]$ be the cell, $X_b$ the points contained in $[0,b)$, and $X'$ the set of points
in $[0,b)$ which are not also contained in $[0,a]$. The points $X'$ must thus be internal
in at least one dimension, so by invariant 1, the points $X'$ are internal in exactly one
dimension $j \in \{1,\ldots,d\}$. Note that if a point $x \in X'$ is internal in
dimension $j$, and hence has $x_i \leq a_i$ for every $i \in \{1,\ldots,d\}$, $i \neq j$, 
then for any $y \in (a,b)$ we have that $x \in [0,y)$ if and only if $y_j > x_j$; that is,
if $x$ is internal in dimension $j$, then the membership of $x$ in $[0,y)$ is determined
only by $y_j$. Now, for each $j \in \{1,\ldots,d\}$, let~$X'_j=\{x \in X': a_j < x_j <
b_j\}$ be the points internal in dimension $j$; note that this partitions $X'$. 
We can then determine $X \cap [0,y)$ for any $y \in (a,b)$ by looking at the coordinates
$y_j$ independently, that is,
\[
[0,y) \cap X = ([0,a] \cap X_b) \cup \bigcup_{j=1}^d \{x \in X'_j: x_j < y_j\}.
\]
This independence makes the problem suitable for the algorithmic technique of \emph{dynamic
  programming} (see, e.g.,~\cite{CLRS09}).
Briefly, let $f_j(y) = |X_b \cap [0,a]| + \sum_{k=1}^j |\{x \in X'_k: x_k < y_k\}|$.
For $i \in \{1,\ldots,n\}$ and $j \in \{1,\ldots,d\}$, 
let $p(i,j)$ be the minimum value of $\prod_{k=1}^j y_k$ such that $f_j(y) \geq i$, and
let $q(i,j)$ be the maximum value of $\prod_{k=1}^j y_k$ such that $f_j(y) \leq i$ and $y \in (a,b)$. 
By sorting the coordinates $x_j$ of every set $X'_j$, it is easy both to compute
the values of $p(\cdot,1)$ and $q(\cdot,1)$, and to use the values of $p(\cdot,j)$ and
$q(\cdot,j)$ to compute the values of $p(\cdot,j+1)$ and $q(\cdot,j+1)$. The maximum
discrepancy attained for a box $[0,y)$ for $y \in (a,b)$ can then be computed from
$p(\cdot,d)$ and $q(\cdot,d)$; note in particular that for $y \in (a,b)$, we have
$f_d(y) = |X \cap [0,y)|$. For details, we refer to \cite{DEM96}.
Discrepancy for boxes $[0,y]$ with $y_j=1$ for one or several $j\in\{1,\ldots,d\}$
can be handled in a similar way. 

Some slight extra care in the analysis of the dynamic programming, and application of a
more intricate form of the decomposition of Overmars and Yap, will lead to a running time
of $O(n^{d/2+1})$ and $O(n)$ space, as shown in~\cite{DEM96}.

\subsection{Complexity Results}
\label{Complexity}

We saw in the previous sections that for fixed dimension $d$ the most efficient
algorithm for calculating the star discrepancy of arbitrary $n$-point sets in 
$[0,1)^d$, the algorithm of Dobkin, Eppstein, and Mitchell, has a running time
of order $O(n^{1+d/2})$. 

So the obvious question is if it is possible to construct a faster algorithm 
for the discrepancy calculation whose running time does \emph{not} depend exponentially  
on the dimension $d$. 

There are two recent complexity theoretical results that suggest that such an algorithm
does not exist---at least not if the famous hypotheses from complexity theory,
namely that P $\neq$ NP and the stronger \emph{exponential time hypothesis} \cite{IP99},
are true.

\subsubsection{Calculating the Star Discrepancy is NP-Hard}
\label{NP}

Looking at identity (\ref{disfor}), we see that 
the calculation of the star discrepancy of a given point set $X$ is
in fact a \emph{discrete optimization problem}, namely the problem to find
an $y\in\overline{\Gamma}$ that maximizes the function value 
\begin{equation*}
 \delta^*(y,X) = \max \left\{ \delta(y,X) \,,\, \overline{\delta}(y,X) \right\}.
\end{equation*}
In \cite{GSW09} it was proved that the calculation of the star discrepancy is
in fact an NP-hard optimization problem. Actually, a stronger statement was proved in 
\cite{GSW09}, namely that the calculation of the star discrepancy is NP-hard even
if we restrict ourselves to the easier sub-problem where all the coordinates of the input
have finite binary expansion, i.e., are of the form $k2^{-\kappa}$ for some $\kappa \in \N$
and some integer $0\le k \le 2^{\kappa}$. To explain this result in more detail let us
start by defining the \emph{coding length} of a real number from the interval $[0,1)$ to be
the number of digits in its binary expansion.

Informally speaking, the class NP is the class of all \emph{decision problems}, i.e., 
problems with a true-or-false answer, for which the instances with an affirmative
answer can be decided in polynomial time by a 
\emph{non-deterministic Turing machine}\footnote{NP stands for ``non-deterministic
polynomial time''.}; here ``polynomial'' means polynomial in the coding length of the input.
Such a non-deterministic Turing machine can be described as consisting of a
non-deterministic part, which generates for a given instance a polynomial-length candidate
(``certificate'') for the solution, and a deterministic part, which verifies in polynomial
time whether this candidate leads to a valid solution.

In general, the NP-hardness of an optimization problem $U$ is proved by
verifying that deciding the so-called \emph{threshold language} of $U$ is 
an NP-hard decision
problem (see, e.g., \cite[Sect.~2.3.3]{Hro03} or, for a more informal explanation,
\cite[Sect.~2.1]{GJ79}). Thus it was actually
shown in \cite{GSW09} that the following decision problem is NP-hard:

\vspace{1ex}

{\noindent
{\bf Decision Problem} {\sc Star Discrepancy}\\
{\it Instance:} Natural numbers $n,d \in \N$, sequence 
$X=(x^{(i)})_{i=1}^n$ in $[0,1)^d,  
\varepsilon \in (0,1]$\\
{\it Question:} Is $d_{\infty}^*(X) \geq \varepsilon$?}
\vspace{1ex}  

Notice that the input of {\sc Star Discrepancy} has only finite coding length if the binary expansion of all 
coordinates and of $\varepsilon$ is finite.
The standard approach to prove NP-hardness of some decision problem is to 
show that another decision problem that is known to be NP-hard can be reduced to it in polynomial
time.
This approach was also used in \cite{GSW09}, where the graph
theoretical problem {\sc Dominating Set} was reduced to {\sc Star Discrepancy}.
The decision problem {\sc Dominating Set} is defined as follows.

\begin{definition}
Let $G=(V,E)$ be a graph, where $V$ is the finite set of vertices and 
$E\subseteq \{\{v,w\} \,|\, v,w\in V\,,\, v\neq w \}$ the set of edges of $G$. 
Let $M$ be a subset of $V$. Then $M$ is called a 
\emph{dominating set of $G$} if for all $v\in V\setminus M$ there 
exists a $w\in M$ such that $\{v,w\}$ is contained in $E$.
\end{definition}

{\noindent
{\bf Decision Problem} {\sc Dominating Set}\\
{\it Instance:} Graph $G=(V,E)$, $m\in \{1,\ldots,|V|\}$\\
{\it Question:} Is there a dominating set $M\subseteq V$ of cardinality at most $m$?}
\vspace{1ex} 

The decision problem {\sc Dominating Set} is well studied in the literature and 
known to be NP-complete, see, e.g., \cite{GJ79}.

Before we explain how to reduce {\sc Dominating Set} to {\sc Star Discrepancy},
let us introduce another closely related decision problem, which will also be important
in Section~\ref{ILP}.

\vspace{1ex}

{\noindent
{\bf Decision Problem} {\sc Empty Half-Open Box}\\
{\it Instance:} Natural numbers $n,d\in \N$, 
                sequence $X=(x^{(i)})^n_{i=1}$ in $[0,1)^d$, 
                $\varepsilon \in (0,1]$
{\it Question:} Is there a $y\in \overline{\Gamma}(X)$ with $A(y,X) = 0$ and
$V_y \ge \varepsilon$?}
\vspace{1ex}

As in the problem {\sc Star Discrepancy} the coding length of the input of {\sc Empty Half-Open Box} is only 
finite if the binary expansion of all coordinates and of $\varepsilon$ is finite. For the reduction of 
{\sc Dominating Set} it is sufficient to consider these instances with finite coding length.
We now explain how to reduce {\sc Dominating Set} to {\sc Empty Half-Open Box} in polynomial time;
this proof step can afterwards be re-used to establish that {\sc Dominating Set}
can indeed be reduced to {\sc Star Discrepancy}.

\begin{theorem}[{\cite[Thm.~2.7]{GSW09}}]\label{EHOB}
The decision problem {\sc Empty Half-Open Box} is NP-hard.
\end{theorem}

\begin{proof}
Let $G=(V,E), m \in \{1,\ldots,|V|\}$ be an instance of {\sc Dominating Set}. 
We may assume that $V= \{1,\dots, n\}$ for $n:=|V|$. 
Let $\alpha, \beta \in [0,1)$ have finite binary expansions and satisfy $\beta < \alpha^n$; for instance,
we may choose $\alpha =1/2$ and $\beta =0$.
For $i,j \in \{1,\ldots,n\}$ put 
\begin{align*}
x^{(i)}_j :=
\begin{cases}
\alpha, &\text{ if } \{i,j\} \in E \text{ or } i=j,\\
\beta, &\text{ otherwise.}
\end{cases}
\end{align*}
and put
$x^{(i)}:=(x^{(i)}_j)_{j=1}^n \in [0,1)^n$ and $X:=(x^{(i)})_{i=1}^n$.
We shall show that there is a dominating set $M \subseteq V$ of cardinality at most $m$ 
if and only if there is a $y \in \overline{\Gamma}(X)$ such that $A(y,X)=0$ and
$V_y \geq \alpha^m$.

Firstly, assume that there is a dominating set $M \subseteq V$ 
of cardinality at most $m$. 
Put 
\begin{align*}
y_j :=
\begin{cases}
\alpha, &\text{ if } j \in M,\\
1, &\text{ otherwise.}
\end{cases}
\end{align*}
and $y:=(y_j)_{j=1}^n$.
Then $y \in \overline{\Gamma}(X)$ and 
$V_y = \alpha^{|M|} \ge \alpha^m$.
Hence it suffices to prove that $[0,y)\cap X = \emptyset$. Now for each $i \in M$ we have $x^{(i)}_i = \alpha$, i.e., 
$x^{(i)} \notin [0,y)$.
For every $i \in V \setminus M$ there is, by definition of a dominating set, 
a $\nu \in M$ such that $\{i,\nu\} \in E$, implying $x^{(i)}_\nu = \alpha$ which
in turn yields $x^{(i)} \notin [0,y)$. Therefore $A(y,X)=0$.

Secondly, assume the existence of a $y \in \overline{\Gamma}(X)$ such that $A(y,X)=0$ and $V_y \geq \alpha^m$. 
Recall that $y \in \overline{\Gamma}(X)$ implies that $y_j\in \{\beta,\alpha,1\}$ for all $j$. 
%
Since $\beta < \alpha^n \le V_y$, we have 
$|\left\{ j \in \{1,\ldots, n\} \,|\, y_j \geq \alpha \right\}| = n$. 
Putting $M:=\{ i \in \{1,\ldots, n\}\, |\, \ y_i = \alpha \}$, we have $|M| \le m$. 
Since $A(y,X)=0$, we obtain $|M| \ge 1$, and for each $i \in \{1,\ldots,n\}$ 
there exists a $\nu \in M$ such that $\{i,\nu\} \in E$ or $i \in M$. Hence $M$ is a dominating set of $G$ with size at most $m$.
\qed
\end{proof}

In \cite{GSW09} further decision problems of the type ``maximal half-open box for 
$k$ points'' and ``minimal closed box for $k$ points'' were studied; these problems
are relevant for an algorithm of E.~Thi\'emard that is based on integer linear programming,
see Section~\ref{ILP}.

\begin{theorem}[{\cite[Thm.~3.1]{GSW09}}]
\label{TheoStar}
{\sc Star Discrepancy} is NP-hard.
\end{theorem}

\emph{Sketch of the proof}.
Due to identity (\ref{disfor}) the decision problem {\sc Star Discrepancy} can be formulated in an equivalent way: Is there a $y \in \overline{\Gamma}(X)$ such 
that $\delta(y,X) \geq \varepsilon$ or 
$\overline{\delta}(y,X) \geq \varepsilon$? 
The NP-hardness of this equivalent formulation
can again be shown by polynomial time reduction from {\sc Dominating Set}. So let 
$V=\{1,\ldots,n\}$, and let $G=(V,E), m \in \{1,\ldots, n\}$ be an 
instance of {\sc Dominating Set}. We may 
assume without loss of generality $n \geq 2$ and $m<n$. Put $\alpha := 1- 2^{-(n+1)}$, $\beta := 0$, and
\begin{align*}
x^{(i)}_j :=
\begin{cases}
\alpha, &\text{ if } \{i,j\} \in E \text{ or } i=j,\\
\beta, &\text{ otherwise.}
\end{cases}
\end{align*}
The main idea is now to prove for $X:=(x^{(i)})_{i=1}^n$ that $d_{\infty}^*(X) \ge \alpha^m =: \varepsilon$
if and only if there is a dominating set $M\subseteq V$ for $G$ with 
$|M| \le m$.

From the proof of Theorem~\ref{EHOB} we know that the existence of such a dominating
set $M$ is equivalent to the existence of a $y \in \overline{\Gamma}(X)$ with $A(y,X) = 0$
and $V_y \ge \alpha^m$.
Since the existence of such a $y$ implies 
$d_{\infty}^*(X) \ge \alpha^m$, it remains only to show that 
$d_{\infty}^*(X) \ge \alpha^m$ implies in turn the existence of such a $y$.
This can be checked with the help of Bernoulli's inequality, for 
the technical details we refer to \cite{GSW09}.

\subsubsection{Calculating the Star Discrepancy is W[1]-Hard}
\label{W[1]}

Although NP-hardness (assuming P $\neq$ NP) excludes a running time of $(n+d)^{O(1)}$ for
computing the star discrepancy of an input of $n$ points in $d$ dimensions, this still
does not completely address our running time concerns. In a nutshell, we know that the
problem can be solved in polynomial time for every fixed $d$ 
(e.g., by the algorithm of Dobkin, Eppstein, and Mitchell), and that it is NP-hard for
arbitrary inputs, but we have no way of separating a running time of $n^{\Theta(d)}$ from,
say, $O(2^dn^2)$, which would of course be a huge breakthrough for computing
low-dimensional star discrepancy. 

The usual framework for addressing such questions is \emph{parameterized complexity}.
Without going into too much technical detail, a parameterized problem is a decision problem whose
inputs are given with a \emph{parameter} $k$. Such a problem is \emph{fixed-parameter
tractable} (FPT) if instances of total length $n$ and with parameter $k$ can be solved in
time $O(f(k)n^c)$, for some constant $c$ and an arbitrary function $f(k)$.  Observe that this
is equivalent to being solvable in $O(n^c)$ time for every fixed $k$, as contrasted to the
previous notion of polynomial time for every fixed $k$, which also includes running times such as $O(n^k)$.
We shall see in this section that, unfortunately, under standard complexity theoretical
assumptions, no such algorithm is possible (in fact, under a stronger assumption, not even
a running time of~$O(f(d)n^{o(d)})$ is possible). 

Complementing the notion of FPT is a notion of parameterized hardness. 
A \emph{parameterized reduction} from a parameterized problem $\cal{Q}$ to a parameterized
problem ${\cal Q'}$ is a reduction which maps an instance~$I$ with parameter $k$ of ${\cal
  Q}$ to an instance $I'$ with parameter $k'$ of ${\cal Q'}$, such that
\begin{enumerate}
\item $(I,k)$ is a true instance of ${\cal Q}$ if and only if $(I',k')$ is a true instance of ${\cal Q'}$,  
\item $k' \leq f(k)$ for some function~$f(k)$, and 
\item the total running time of the reduction is FPT (i.e., bounded by $O(g(k)||I||^c)$,
for some function $g(k)$ and constant $c$, where~$||I||$ denotes the total coding length of the input).
\end{enumerate}
It can be verified that such a reduction and an FPT-algorithm for ${\cal Q'}$ imply an
FPT-algorithm for ${\cal Q}$. 

The basic hardness class of parameterized complexity, analogous to the class NP of
classical computational complexity, is known as W[1], and can be defined as follows. 
Given a graph $G=(V,E)$, a \emph{clique} is a set $X \subseteq V$ such that for any $u,v
\in X$, $u \neq v$, we have $\{u,v\} \in E$. 
Let \textsc{$k$-Clique} be the following parameterized problem. 

\vspace{1ex}
{\noindent
{\bf Parameterized Problem} {\sc $k$-Clique}\\
{\it Instance:} A graph $G=(V,E)$; an integer $k$. \\
{\it Parameter:} $k$ \\
{\it Question:} Is there a clique of cardinality $k$ in $G$?
}
\vspace{1ex}

The class W[1] is then the class of problems reducible to \textsc{$k$-Clique} under
parameterized reductions, and a parameterized problem is W[1]-hard if there is a
parameterized reduction to it from \textsc{$k$-Clique}. (Note that, similarly, the
complexity class NP can be defined as the closure of, e.g.,
\textsc{Clique} or \textsc{3-SAT} under standard polynomial-time reductions.)
The basic complexity assumption of parameterized complexity is that FPT $\neq$ W[1], or
equivalently, that the \textsc{$k$-Clique} problem is not fixed-parameter tractable; this
is analogous to the assumption in classical complexity that P $\neq$ NP.

For more details, see the books of R.~G.~Downey and M.~R.~Fellows~\cite{DF99} or J.~Flum and
M.~Grohe~\cite{FlumGroheBook}. In particular, we remark that there is a different definition
of W[1] in terms of problems solvable by a class of restricted circuits (in fact, W[1] is
just the lowest level of a hierarchy of such classes, the so-called \emph{W-hierarchy}),
which arguably makes the class definition more natural, and the conjecture FPT $\neq$ W[1]
more believable. 

P.~Giannopoulus et al.~\cite{GKWW12} showed the following.

\begin{theorem}[\cite{GKWW12}]
There is a polynomial-time parameterized reduction from \textsc{$k$-Clique} with parameter
$k$ to \textsc{Star Discrepancy} with parameter $d=2k$.
\end{theorem}

Thus, the results of~\cite{GKWW12} imply that \textsc{Star Discrepancy} has no algorithm
with a running time of $O(f(d)n^c)$ for any function~$f(d)$ and constant~$c$, unless FPT
$=$ W[1]. 

A stronger consequence can be found by considering the so-called exponential-time
hypothesis (ETH). This hypothesis, formalized by Impagliazzo and Paturi~\cite{IP99},
states that \textsc{3-SAT} on $n$ variables cannot be solved in time $O(2^{o(n)})$; in a
related paper~\cite{IPZ01}, this was shown to be equivalent to similar statements about
several other problems, including that \textsc{$k$-SAT} on $n$ variables cannot be solved
in time $O(2^{o(n)})$ for $k\geq 3$ and that \textsc{3-SAT} cannot be solved in time $O(2^{o(m)})$,
where~$m$ equals the total number of clauses in an instance (the latter form is
particularly useful in reductions).

It has been shown by J.~Chen et al.~\cite{ChenCFHJKX05,ChenHKX06} that \textsc{$k$-Clique}
cannot be solved in $f(k)n^{o(k)}$ time, for any function~$f(k)$, unless ETH is false. 
We thus get the following corollary.

\begin{corollary}[\cite{GKWW12}]
\textsc{Star Discrepancy} for $n$ points in $d$ dimensions cannot be solved exactly in
$f(d)n^{o(d)}$ time for any function~$f(d)$, unless ETH is false.
\end{corollary}


In fact, it seems that even the constant factor in the exponent of the running time of the
algorithm of Dobkin, Eppstein and Mitchell~\cite{DEM96} would be very difficult to improve;
e.g., a running time of $O(n^{d/3+c})$ for some constant~$c$ would imply a new faster
algorithm for \textsc{$k$-Clique}. (See~\cite{GKWW12} for more details, and for a
description of the reduction.)

\subsection{Approximation of the Star Discrepancy}
\label{approximation}

As seen in Section~\ref{NP} and~\ref{W[1]}, complexity theory tells us that the
exact calculation of the star discrepancy of large point sets in high dimensions is infeasible.

The theoretical bounds for the star discrepancy of low-discrepancy point sets 
that are available in the literature  describe the 
asymptotic behavior of the star discrepancy well only if the number of points $n$ tends to infinity. These 
bounds are typically useful for point sets of size $n \gg e^d$, but give no helpful information
for moderate values of $n$. To give concrete examples, we restate here some numerical results
provided in \cite{Joe12} for bounds based on inequalities of \emph{Erd\H os-Tur\'an-Koksma-type} (see, e.g., \cite{DP10, Nie92}).

If the point set $P_n(z)$ is an $n$-point rank-$1$ lattice in $[0,1]^d$ with generating
vector $z\in \mathbb{Z}^d$ (see, e.g., \cite{Nie92, SJ94}), then its discrepancy can be bounded by
\begin{equation}
\begin{split}
 d^*_\infty(P_n(z)) &\le 1 - \left( 1 - \frac{1}{n} \right)^d + T(z,n) \le 1 - \left( 1 - \frac{1}{n} \right)^d + W(z,n)\\
&\le 1 - \left( 1 - \frac{1}{n} \right)^d + R(z,n)/2; 
\end{split}
\end{equation}
here the quantities $W(z,n)$ and $R(z,n)$ can be calculated to a fixed precision in $O(nd)$ operations \cite{Joe12, JS92},
while the calculation of $T(z,n)$ requires $O(n^2 d)$ operations (at least there is so far no faster algorithm known).

S.~Joe presented in \cite{Joe12} numerical examples where he calculated the values of $T(z,n)$, $W(z,n)$, and $R(z,n)$ for 
generators $z$ provided by a \emph{component-by-component algorithm} (see, e.g., \cite{Kuo03, NC06, SKJ02, SR02}) 
from \cite[Sect. 4]{Joe12}. For dimension $d=2$ and $3$ he was able to compare the quantities with the exact values
\begin{equation}
 d^*_\infty(P_n(z)) - \left[ 1 - \left( 1 - \frac{1}{n} \right)^d \right] =: E(z,n).
\end{equation}
In $d=2$ for point sets ranging from $n=157$ to $10,007$, the smallest of the three quantities, $T(z,n)$, was $8$ to $10$ times larger than
$E(z,n)$. In dimension $d=3$ for point sets ranging from $n=157$ to $619$ it was even more than $18$ times larger. (Joe used for the computation
of $E(z,n)$ the algorithm proposed in \cite{BZ93} and was therefore limited to this range of examples.)

Computations for $d=10$ and $20$ and $n=320,009$ led to $T(z,n) = 1.29 \cdot 10^4$ and $5.29 \cdot 10^{13}$, respectively.
Recall that the star discrepancy and $E(z,n)$ are always bounded by $1$. 
The more efficiently computable quantities $W(z,n)$ and $R(z,n)$ led obviously to worse results, but $W(z,n)$ was at least very close
to $T(z,n)$. 

This example demonstrates that for good estimates of the star discrepancy for point sets of practicable size we can unfortunately not rely
on theoretical bounds.

Since the exact calculation of the star discrepancy is infeasible in high dimensions, the only remaining alternative is to consider approximation
algorithms. In the following we present the known approaches.

\subsubsection{An Approach Based on Bracketing Covers}
\label{bracketing}

An approach that approximates the star discrepancy of a given set $X$ up to a
user-specified error $\delta$ was presented by E.~Thi\'emard \cite{Thi00, Thi01a}.
It is in principle based on the generation of suitable
\emph{$\delta$-bracketing covers} (which were not named in this way in 
\cite{Thi00, Thi01a}).
Let us describe Thi\'emard's approach in detail.

The first step is to 
``discretize'' the star discrepancy at the cost of an approximation
error of at most $\delta$. The corresponding discretization is different from the one described
in Section~\ref{EAC}; in particular, it is completely independent of the input
set $X$. The discretization is done by choosing a 
suitable finite set of test boxes anchored in zero whose upper right 
corners form a so-called \emph{$\delta$-cover}. We repeat
here the definition from \cite{DGS05}. 

\begin{definition}
A finite subset $\Gamma$ of $[0,1]^d$ is called a \emph{$\delta$-cover}
of the class $\mathcal{C}_d$ of all axis-parallel half-open boxes anchored in zero 
(or of $[0,1]^d$) if for all $y\in [0,1]^d$ there exist
$x,z\in \Gamma\setminus \{0\}$ such that 
\begin{equation*}
x\le y\le z
\hspace{3ex}\text{and}\hspace{3ex} 
V_z - V_x \le \delta.
\end{equation*}
Put 
\begin{equation*}
N(\mathcal{C}_d, \delta) := \min\{|\Gamma| \,|\,\, \Gamma\hspace{1ex} 
\text{$\delta$-cover of $\mathcal{C}_d$.}\}
\end{equation*}
\end{definition}

Any $\delta$-cover $\Gamma$ of $\mathcal{C}_d$ satisfies the following 
approximation property: 

\begin{lemma}
\label{Approx}
Let $\Gamma$ be a $\delta$-cover of $\mathcal{C}_d$.
For all finite sequences $X$ in $[0,1]^d$
we have
\begin{equation}
\label{approx}
d^*_\infty(X) \le d^*_{\Gamma}(X) + \delta,
\end{equation}
where
\begin{equation*}
d^*_{\Gamma}(X) := \max_{y\in\Gamma} \left| V_y
- A(y,X) \right|
\end{equation*}
can be seen as a discretized version of the star discrepancy. 
\end{lemma}

A proof of Lemma~\ref{Approx} is straightforward. (Nevertheless, it is, e.g., 
contained in \cite{DGS05}.)

In \cite{Gne08a} the notion of $\delta$-covers was related to the 
concept of bracketing entropy, which is well known in the theory
of empirical processes. We state here the definition for the set
system $\mathcal{C}_d$ of anchored axis-parallel boxes (a general definition
can, e.g., be found in \cite[Sect. 1]{Gne08a}):

\begin{definition}
A closed axis-parallel box $[x,z]\subset [0,1]^d$ is a 
\emph{$\delta$-bracket}
of $\mathcal{C}_d$ if $x\le z$ and 
$V_z-V_x \le \delta$. A 
\emph{$\delta$-bracketing cover} of $\mathcal{C}_d$ is a set of 
$\delta$-brackets whose union is 
$[0,1]^d$. By $N_{[\,]}(\mathcal{C}_d,\delta)$ we denote the 
\emph{bracketing number} of $\mathcal{C}_d$ (or of $[0,1]^d$), i.e., the smallest number
of $\delta$-brackets whose union is $[0,1]^d$. The quantity 
$\ln  N_{[\,]}(\mathcal{C}_d,\delta)$ is called the \emph{bracketing entropy}.
\end{definition}

The bracketing number and the quantity $N(d,\delta)$ are related to
the \emph{covering} and the \emph{$L_1$-packing number}, see, e.g., \cite[Rem.~2.10]{DGS05}. 

It is not hard to verify that 
\begin{equation}
\label{relbra}
N(\mathcal{C}_d,\delta) \le 2 N_{[\,]}(\mathcal{C}_d,\delta) \le N(\mathcal{C}_d,\delta)
(N(\mathcal{C}_d, \delta)+1)
\end{equation}
holds. Indeed, if $\mathcal{B}$ is a $\delta$-bracketing cover, then it is easy to see that
\begin{equation}
\label{gamma_b}
\Gamma_{\mathcal{B}} := \{x\in [0,1]^d\setminus \{0\} \,|\,
\exists y\in [0,1]^d: [x,y]\in\mathcal{B}\hspace{1ex}\text{or}\hspace{1ex}
[y,x]\in\mathcal{B} \}
\end{equation}
is a $\delta$-cover. If $\Gamma$ is a $\delta$-cover, then
\begin{equation*}
\mathcal{B}_\Gamma := \{[x,y] \,|\, x,y\in \Gamma\cup \{0\} \,,\, 
[x,y] \hspace{1ex}\text{is a $\delta$-bracket}\,,\, x\neq y\}
\end{equation*}
is a $\delta$-bracketing cover. These two observations imply (\ref{relbra}).

In \cite{Gne08a} it is shown that
\begin{equation}
\label{brac}
\delta^{-d}(1-O_d(\delta)) \le N_{[\,]}(\mathcal{C}_d,\delta) \le
2^{d-1} (2\pi d)^{-1/2}{\rm e}^d (\delta^{-1} + 1)^d,
\end{equation}
see \cite[Thm. 1.5 and 1.15]{Gne08a}.
The construction that leads to the upper bound in (\ref{brac})
implies also 
\begin{equation}
\label{braasym}
N_{[\,]}(\mathcal{C}_d, \delta) \le (2\pi d)^{-1/2}{\rm e}^d \delta^{-d} + 
O_d(\delta^{-d+1})
\end{equation}
(see \cite[Remark 1.16]{Gne08a}) and 
\begin{equation}
\label{good}
N(\mathcal{C}_d,\delta) \le 2^d (2\pi d)^{-1/2} {\rm e}^d (\delta^{-1}+1)^d.
\end{equation}

For more information about $\delta$-covers and $\delta$-bracketing covers
we refer to the original articles \cite{DGS05, Gne08a, Gne08b} and the 
survey article \cite{Gne12b}; \cite{Gne08b} and \cite{Gne12b} contain also
several figures showing explicit two-dimensional constructions.

The essential idea of Thi\'emard's algorithm from \cite{Thi00, Thi01a}
is to generate for a given point set $X$ and a user-specified error 
$\delta$ a small $\delta$-bracketing cover 
$\mathcal{B} =\mathcal{B}_\delta$ of $[0,1]^d$ and to approximate 
$d^*_\infty(X)$ by $d^*_{\Gamma}(X)$, where $\Gamma = \Gamma_{\mathcal{B}}$ as in 
(\ref{gamma_b}), up to an error of at most $\delta$,
see Lemma~\ref{Approx}.

The costs of generating $\mathcal{B}_\delta$
are of order $\Theta(d|\mathcal{B}_\delta|)$. If we count the 
number of points in $[0,y)$ for each
$y\in\Gamma_{\mathcal{B}}$ in a naive way, this results in an
overall running time of $\Theta(dn|\mathcal{B}_\delta|)$ for the
whole algorithm. As Thi\'emard pointed out in \cite{Thi01a}, this
\emph{orthogonal range counting} can be done in moderate dimension $d$ 
more effectively by 
employing data structures based on \emph{range trees}, see, e.g., 
\cite{BCKO08, Meh84}. 
This approach reduces in moderate dimension $d$ the time 
$O(dn)$ per test box that is needed for the naive counting to 
$O(\log^d n)$. 
Since a range tree for $n$ points can be
generated in $O(C^d n\log^d n)$ time, $C>1$ some constant,
this results in an overall running time of
\begin{equation*}
O((d+\log^d n) |\mathcal{B}_\delta| + C^d n \log^d n)\,.
\end{equation*}
As this approach of orthogonal range counting is obviously not very beneficial
in higher dimension (say, $d>5$), we do not further explain it here, but refer for 
the details to \cite{Thi01a}. 

The smallest bracketing covers $\mathcal{T}_\delta$ used by Thi\'emard can be found
in \cite{Thi01a}; they differ from the constructions provided in \cite{Gne08a, Gne08b},
see \cite{Gne08b, Gne12}.
He proved for his best constructions the upper bound
\begin{equation*}
|\mathcal{T}_\delta| \leq
e^d \left( \frac{\ln \delta^{-1}}{\delta} 
+ 1 \right)^d\,,
\end{equation*}
a weaker bound than (\ref{braasym}) and (\ref{good}), which both hold for the 
$\delta$-bracketing covers constructed in \cite{Gne08a}.
Concrete comparisons of Thi\'emard's bracketing
covers with other constructions in dimension $d=2$ can be found in
\cite{Gne08b}, where also optimal two-dimensional bracketing covers are 
provided.

The lower bound in (\ref{brac}) proved in \cite{Gne08a} immediately 
implies a lower bound for the running time of Thi\'emard's algorithm,
regardless how cleverly the $\delta$-bracketing covers are chosen.
That is because the dominating factor in the running time is the construction
of the $\delta$-bracketing cover $|\mathcal{B}_\delta|$, which
is of order $\Theta(d|\mathcal{B}_\delta|)$. Thus (\ref{brac})
shows that the running time of the algorithm is exponential in $d$.
(Nevertheless smaller $\delta$-bracketing covers, which may,
e.g., be generated by extending the ideas from \cite{Gne08b} to
arbitrary dimension $d$, would widen the range of applicability of 
Thi\'emard's algorithm.)
Despite its limitations, Thi\'emard's algorithm is a helpful tool
in moderate dimensions, as was reported, e.g., in \cite{DGW10, Thi00,Thi01a} or \cite{OSG12},
see also Section~\ref{QUALTEST}.

For more specific details we refer to \cite{PC05, Thi00, Thi01a}.
For a modification of Thi\'emard's approach to approximate 
$L_\infty$-extreme discrepancies see \cite[Sect.~2.2]{Gne08a}.

\subsubsection{An Approach Based on Integer Linear Programming}
\label{ILP}

Since the large scale enumeration problem (\ref{disfor}) is infeasible in high dimensions, a number of algorithms have been developed that are based on heuristic approaches. 
One such approach was suggested by E.~Thi\'emard in~\cite{Thi01b} (a more detailed description of his algorithm can be found in his PhD thesis~\cite{Thi00PhD}). 
Thi\'emard's algorithm is based on \emph{integer linear programming}, a concept that we shall describe below. 
His approach is interesting in that, despite being based on heuristics in the initialization phase, it allows for an arbitrarily good approximation of the star discrepancy of any given point set. 
Furthermore, the user can decide on the fly which approximation error he is willing to tolerate. This is possible because the algorithm outputs, during the optimization of the star discrepancy approximation, upper and lower bounds for the exact $d^*_{\infty}(X)$-value.  
The user can abort the optimization procedure once the difference between the lower and upper bound are small enough for his needs, or he may wait until the optimization procedure is finished, and the exact star discrepancy value $d^*_{\infty}(X)$ is computed. 

Before we describe a few details of the algorithm, let us mention that numerical tests in~\cite{Thi01a, Thi01b} 
suggest that the algorithm
from \cite{Thi01b}     outperforms the one from \cite{Thi01a} (see Section~\ref{bracketing}  
for a description of the latter algorithm). 
In particular, instances that are 
infeasible for the algorithm from~\cite{Thi01a} can be solved using the integer linear programming approach described below, see also the discussion in Section~\ref{QUALTEST}.

The basic idea of Thi\'emard's algorithm is to split optimization problem (\ref{disfor}) into $2n$ optimization problems similarly as done in equation (\ref{discrit2}), and to transform these problems into integer linear programs. 
To be more precise, he considers for each value $k \in \{0,1,2,\ldots, n\}$ the volume of the smallest and the largest box containing exactly $k$ points of $X$. These values are denoted $V^k_{\min}$ and $V^k_{\max}$, respectively. 
It is easily verified, using similar arguments as in Lemma~\ref{Disfor}, that these values (if they exist) are obtained by the grid points of $X$, i.e., there exist grid points $y^{k}_{\min} \in \Gamma(X)$, $y^{k}_{\max} \in \overline{\Gamma}(X)$ such that 
\begin{align*}
& \overline{A}(y^{k}_{\min}, X)=k \text{ and } V_{y^{k}_{\min}}=V^{k}_{\min}\,,\\
& A(y^{k}_{\max}, X)=k \text{ and } V_{y^{k}_{\max}}=V^{k}_{\max}\,.
\end{align*}
It follows from (\ref{discrit2}) that 
\begin{align}
\label{thiebound}
d^*_{\infty}(X) = \max_{k \in \{0,1,2,\ldots, n\}} \max
\left\{  
\frac{k}{n}- V^{k}_{\min},
V^{k}_{\max} - \frac{k}{n}
\right\}\,.
\end{align}
As noted in~\cite{GSW09}, boxes containing \emph{exactly} $k$ points may not exist, even if the $n$ points are pairwise different. However, they do exist if for at least one dimension $j \in \{1,\ldots, d\}$ the coordinates $(x_j)_{x \in X}$ are pairwise different. 
If they are pairwise different for all $j \in \{1,\ldots, d\}$, in addition we have 
\begin{align*}
V^1_{\min} \leq \ldots \leq V^n_{\min}
\quad \text{ and } \quad
V^0_{\max} \leq \ldots \leq V^{n-1}_{\max}.
\end{align*}
Note that we obviously have $V^0_{\min}=0$ and $V^n_{\max}=1$. We may therefore disregard these two values in equation (\ref{thiebound}).

As mentioned in Section~\ref{NP} (cf. Theorem~\ref{EHOB} and the text thereafter), 
already the related problem 
``Is $V^0_{\max} \geq \epsilon$?'' is an NP-hard one.
By adding ``dummy points'' at or close to the origin $(0, \ldots, 0)$ one can easily generalize this result to all questions of the type ``Is $V^k_{\max} \geq \epsilon$?''.
Likewise, it is shown in~\cite{GSW09} that the following decision problem is NP-hard. 
\vspace{1ex}

\noindent
{\bf Decision Problem} {\sc $V^k_{\min}$-Box}\\
{\it Instance:} Natural numbers $n,d \in \N$, $k \in \{0,1,\ldots, n\}$, sequence 
$X=(x^{(i)})_{i=1}^n$ in $[0,1)^d,  
\varepsilon \in (0,1]$\\
{\it Question:} Is there a point $y \in \Gamma(X)$ such that 
$\overline{A}(y,X)\geq k$ and 
$V_y \leq \varepsilon$?
\vspace{1ex} 

This suggests that the $V^k_{\max}$- and $V^k_{\min}$-problems are difficult to solve to optimality. 
As we shall see below, in Thi\'emard's integer linear programming ansatz, we will not have to solve all $2n$ optimization problems in (\ref{thiebound}) to optimality. Instead it turns out that for most practical applications of his algorithms only very few of them need to be solved exactly, whereas for most of the problems it suffices to find good upper and 
lower bounds. 
We shall discuss this in more detail below. 

For each of the $2n$ subproblems of computing $V^k_{\min}$ and $V^{k}_{\max}$, respectively, Thi\'emard formulates, by taking the logarithm of the volumes, an integer linear program with $n+d(n-k)$ binary variables. The size of the linear program is linear in the size $nd$ of the input $X$. 
We present here the integer linear program (ILP) for the $V^k_{\min}$-problems. The ones for $V^k_{\max}$-problems are similar. 
However, before we are ready to formulate the ILPs, we need to fix some notation.

For every $n \in \N$ we abbreviate by $S_n$ the set of permutations of $\{1, \ldots, n\}$. 
For all $j \in \{1,\ldots,d\}$ put $x_j^{(n+1)} := 1$ and let $\sigma_j \in S_{n+1}$ such that
\begin{align*}
x_j^{(\sigma_j(1))} \leq \ldots \leq x_j^{(\sigma_j(n))} \leq x_j^{(\sigma_j(n+1))}=1.
\end{align*}
With $\sigma_j$ at hand, we can define, for every index 
$\delta = (\delta_1,\ldots,\delta_d) \in \{1,\ldots,n+1\}^d$ 
the closed and half-open boxes \emph{induced by} $\delta$, 
\begin{align*}
[0,\delta] := \prod_{j=1}^d [0,x_j^{(\sigma_j(\delta_j))}] 
\quad \mbox{ and } \quad 
[0,\delta) \ = \prod_{j=1}^d [0,x_j^{(\sigma_j(\delta_j))}).
\end{align*}
One of the crucial observations for the formulation of the ILPs is the fact that 
$x^{(i)} \in [0,\delta]$ (resp. $x^{(i)} \in [0,\delta)$), 
if and only if for all $j \in \{1,\ldots,d\}$ it holds that 
$ \sigma_j^{-1}(i) \leq \delta_j$ 
(resp. $\sigma_j^{-1}(i) \leq \delta_j - 1$). 
We set
\begin{align*}
z^i_j (\delta) := 
\begin{cases}
1, &\text{ if } \sigma_j^{-1}(i) \leq \delta_j,\\
0, &\text{ otherwise.}
\end{cases}
\end{align*}
Every $\delta$ induces exactly one sequence $z=((z^{(i)}_j(\delta)))_{j=1}^d)_{i=1}^n$ in $(\{0,1\}^{d})^n$, 
and, likewise, for every \emph{feasible} sequence $z$ 
there is exactly one 
$\delta(z) \in \{1,\ldots,n+1\}^d$ with $z=z(\delta(z))$. 
In the following linear program formulation we introduce also the variables $y^{(1)}, \ldots, y^{(n)}$, and we shall have $y^{(i)} = 1$ if and only if $x^{(i)} \in [0,\delta(z)]$.

The integer linear program for the $V^k_{\min}$-problem can now be defined as follows.
\begin{align}
\label{LPmin}
\ln(V^k_{\min}) = \min \ \sum_{j=1}^d [ \ln(x_j^{(\sigma_j(1))}) + \sum_{i=2}^n z^{(\sigma_j(i))}_j (\ln(x_j^{(\sigma_j(i))}) - \ln(x_j^{(\sigma_j(i-1))})) ] 
\end{align}
 \nopagebreak
subject to \nopagebreak
\begin{longtable}{cll}
 (i)& $1=z_j^{(\sigma_j(1))} = \ldots = z_j^{(\sigma_j(k))} 
 \geq \ldots \geq z_j^{(\sigma_j(n))}$ &  $\forall j \in \{1,\ldots,d\}$\\
(ii)& $z_j^{(\sigma_j(i))} = z_j^{(\sigma_j(i+1))}$ & $\forall j \in \{1,\ldots,d\} \forall i \in \{1,\ldots,n\}:$ \\
& & \quad $x_j^{(\sigma_j(i))} = x_j^{(\sigma_j(i+1))}$\\
(iii)& $y^{(i)} \leq z^{(i)}_j$ & $\forall j \in \{1,\ldots,d\} \forall i \in \{1,\ldots,n\}$\\
(iv) & $y^{(i)} \geq 1 - d + \sum_{j=1}^d z^{(i)}_j$ & $\forall i \in \{1,\ldots,n\}$\\
(v)  & $\sum_{i=1}^n y^{(i)} \geq k$ & \\
(vi) & $y^{(i)} \in \{0,1\}$ & $\forall i \in \{1,\ldots,n\}$ \\
(vii)& $z^{(i)}_j \in \{0,1\}$ & $\forall j \in \{1,\ldots,d\} \forall i \in \{1,\ldots,n\}$
\end{longtable}

We briefly discuss the constraints of the integer linear program (\ref{LPmin}).
\begin{itemize}
	\item Since we request at least $k$ points to lie in the box $[0,\delta(z)]$, 
	the inequality $x_j^{(\sigma_j(\delta(z)_j))} \geq x_j^{(\sigma_j(k))}$ must hold for all $j \in \{1,\ldots,d\}$. We may thus fix the values 
	$1 = z_j^{(\sigma_j(1))} = \ldots = z_j^{(\sigma_j(k))}$.
	\item The second constraint expresses that for two points with the same coordinate $x_j^{(\sigma_j(i))} = x_j^{(\sigma_j(i+1))}$ in the $j$th dimension, we must satisfy $z_j^{(\sigma_j(i))} = z_j^{(\sigma_j(i+1))}$. 
	\item The third and fourth condition say that $y^{(i)} = 1$ if and only if $x^{(i)} \in [0,\delta(z)]$. 
	For $x^{(i)} \in [0,\delta(z)]$ we have $\sigma_j^{-1}(i) \leq \delta(z)_j$ and thus  $z^{(i)}_j = 1$, $j \in \{1,\ldots,d\}$. 
	According to condition (iv) this implies $y^{(i)} \geq 1 - d + \sum_{j=1}^d z^{(i)}_j = 1$, and thus $y^{(i)} =1$.
	If, on the other hand, $x^{(i)} \notin [0,\delta(z)]$, there exists a coordinate $j \in \{1,\ldots,d\}$ with $x^{(i)}_j  > \delta(z)_j$. Thus, $z^{(i)}_j = 0$ and condition (iii) implies $y^{(i)} \leq z^{(i)}_j = 0$.
	\item Condition (v) ensures the existence of at least $k$ points inside $[0,\delta(z)]$. 
	\item Conditions (vi) and (vii) are called the \emph{integer} or \emph{binary constraints}. Since only integer (binary) values are allowed, the linear program (\ref{LPmin}) is called a \emph{(binary) integer linear program}. 
	We shall see below that by changing these conditions to $y^{(i)} \in [0,1]$ and $z^{(i)}_j \in [0,1]$, we get the \emph{linear relaxation} of the integer linear program (\ref{LPmin}). The solution of this linear relaxation is a lower bound for the true
 $V^k_{\min}$-solution.
\end{itemize}

Using the $V^k_{\min}$- and $V^k_{\max}$-integer linear programs, 
Thi\'emard computes the star discrepancy of a set $X$ in a sequence of optimization steps, each of which possibly deals with a different $k$-box problem. 
Before the optimization phase kicks in, there is an 
\emph{initialization phase}, in which 
for each $k$ an upper bound 
$\overline{V}^k_{\min}$ for $V^k_{\min}$ and a lower bound 
$\underline{V}^k_{\max}$ for $V^k_{\max}$ is computed. 
This is done by a simple greedy strategy followed by a local optimization procedure that helps to improve the initial value of the greedy strategy. 
Thi\'emard reports that the estimates obtained for the $V^k_{\max}$-problems are usually quite good already, whereas the estimates for the $V^k_{\min}$-problems are usually too pessimistic. 
A lower bound $\underline{V}^k_{\min}$ for the $V^k_{\min}$-problem 
is also computed, using the simple observation that for each dimension $j\in \{1,\ldots, d\}$, the $j$th coordinate of the smallest $V^k_{\min}$-box must be at least as large as the $k$th smallest coordinate in $(x_j)_{x \in X}$. 
That is, we have $V^k_{\min} \geq \prod_{j=1}^{d}{x_j^{(\sigma_j(k))}}$. We initialize the lower bound $\underline{V}^k_{\min}$ by setting it equal to this expression. 
Similarly, $\prod_{j=1}^{d}{x_j^{(\sigma_j(k+1))}}$ is a lower bound for the $V^k_{\max}$-problem, but this bound is usually much worse than the one provided by the heuristics.
For an initial upper bound of $V^k_{\max}$, Thi\'emard observes that the $V^{n-1}_{\max}$-problem can be solved easily. 
In fact, we have $V^{n-1}_{\max}=\max \{ x_j^{(\sigma_j(n))} \mid j \in \{1,\ldots, d\} \}$. 
As mentioned in the introduction to this section, if for all $j \in \{1, \ldots, d \}$ the $j$th coordinates $x^{(1)}_j, \ldots, x^{(n)}_j$ of the points in $X$ are pairwise different, we have 
$V^0_{\max} \leq \ldots \leq V^{n-1}_{\max}$. 
Thus, in this case, we may initialize $\overline{V}^k_{\max}$, $k=1, \ldots, n-1$, by $V^{n-1}_{\max}$.

From these values 
(we neglect a few minor steps in Thi\'emard's computation) 
we compute the following estimates 
\begin{align*}
& \underline{D}^k_{\min}(X) := \tfrac{k}{n}-\overline{V}^k_{\min}
\text{ and } 
\overline{D}^k_{\min}(X) := \tfrac{k}{n}-\underline{V}^k_{\min}\,,\\
& \underline{D}^k_{\max}(X) := \underline{V}^k_{\max} - \tfrac{k}{n} \text{ and }
\overline{D}^k_{\max}(X) := \overline{V}^k_{\max}  - \tfrac{k}{n}\,.\\
\end{align*}
Clearly, 
$\underline{D}^k_{\min}(X) \leq \tfrac{k}{n}-V^k_{\min} \leq \overline{D}^k_{\min}(X)$ 
and 
$\underline{D}^k_{\max}(X) \leq V^k_{\max}-\tfrac{k}{n} \leq \overline{D}^k_{\max}(X)$.

After this initialization phase, the \emph{optimization phase} begins. It proceeds in rounds. In each round, the $k$-box problem yielding the largest estimate 
$$
\overline{D}^*_{\infty}(X) :=
\max
\left\{ 
\max_{k \in \{1,2,\ldots, n\}} 
{\overline{D}^k_{\min}(X)}, 
\max_{k \in \{0,1,\ldots, n-1\}} 
{\overline{D}^k_{\max}(X)}
\right\}
$$
is investigated further. 

If we consider a $V^k_{\min}$- or a $V^k_{\max}$-problem for the first time, we regard the \emph{linear relaxation} of the integer linear program for $\ln(V^k_{\min})$ or $\ln(V^k_{\max})$, respectively. 
That is---cf. the comments below the formulation of the ILP for $V^k_{\min}$-problem above---instead of requiring the variables $y^{(i)}$ and $z^{(i)}_j$ to be either $0$ or $1$, we only require them to be in the interval $[0,1]$. 
This turns the integer linear program into a linear program. 
Although it may seem that this relaxation does not change much, linear programs are known to be polynomial time solvable, and many fast readily available solving procedures, e.g., commercial tools such as CPLEX, are available. 
Integer linear programs and binary integer linear programs such as ours, on the other hand, are known to be NP-hard in general, and are usually solvable only with considerable computational effort.
Relaxing the binary constraints (vi) and (vii) in (\ref{LPmin}) can thus be seen as a heuristic to get an initial approximation of the $V^k_{\min}$- and $V^k_{\max}$-problems, respectively.

In case of a $V^k_{\min}$-problem the value of this relaxed program is 
a lower bound for $\ln(V^k_{\min})$---we thus obtain new estimates for the two values $\underline{V}^k_{\min}$ and $\overline{D}^k_{\min}$.
If, on the other hand, we regard a $V^k_{\max}$-problem, the solution of the relaxed linear program establishes an upper bound for $\ln(V^k_{\max})$; and we thus get new estimates for $\overline{V}^k_{\max}$ and $\overline{D}^k_{\max}$. 
We may be lucky that we get an integral solution, in which case we have determined $V^k_{\min}$ or $V^k_{\max}$, respectively, and do not need to consider this problem in any further iteration of the algorithm. 

If we consider a $V^k_{\min}$- or $V^k_{\max}$-problem for the second time, we solve the integer linear program itself, using a standard \emph{branch and bound} technique. 
Branch and bound resembles a divide and conquer approach: the problem is divided into smaller subproblems, for each of which upper and lower bounds are computed. 

Let us assume that we are, for now, considering a fixed $V^k_{\min}$-problem ($V^k_{\max}$-problems are treated the same way). As mentioned above, we divide this problem into several subproblems, and we compute upper and lower bounds for these subproblems. 
We then investigate the most ``promising'' subproblems 
(i.e., the ones with the largest upper bound and smallest lower bound for the value of $V^k_{\min}$)
further, until the original $V^k_{\min}$-problem at hand has been solved to optimality or until the bounds for $V^k_{\min}$ are good enough to infer that this $k$-box problem does not cause the maximal discrepancy value in (\ref{thiebound}). 

A key success factor of the branch and bound step is a further strengthening of the integer linear program at hand. 
Thi\'emard introduces further constraints to the ILP, some of which are based on straightforward combinatorial properties of the $k$-box problems and others which are based on more sophisticated techniques such as \emph{cutting planes} and \emph{variable forcing} (cf. Thi\'emard's PhD thesis~\cite{Thi00PhD} for details). 
These additional constraints and techniques strengthen the ILP in the sense that the solution to the linear relaxation is closer to that of the integer program. 
Thi\'emard provides some numerical results indicating that 
these methods frequently yield solutions 
based on which we can exclude the $k$-box problem at hand from our considerations for optimizing (\ref{thiebound}). 
That is, only few of the $2n$ many $k$-box problems in (\ref{thiebound}) need to be solved to optimality, cf.~\cite{Thi01b} for the details. 

As explained above, Thi\'emard's approach 
computes upper and lower bounds for the star discrepancy of a given point set at the same time. 
Numerical experiments indicate that the lower bounds are usually quite strong from the beginning, whereas the initial upper bounds are typically 
too large, and decrease only slowly during the optimization phase, cf.~\cite[Section 4.2]{Thi01b} for a representative graph of the convergence behavior. Typical running times of the algorithms can be found in~\cite{Thi01b} and in~\cite{GWW12}. The latter report contains also a comparison to the alternative approach described in the next section.

\subsubsection{Approaches Based on Threshold-Accepting}
\label{threshold}

In the next two sections we describe three heuristic approaches to compute lower bounds for the star discrepancy of a given point set $X$. 
All three algorithms are based on randomized local search heuristics; two of them on a so-called \emph{threshold accepting approach}, see this section, and one of them on a \emph{genetic algorithm}, see Section~\ref{GenAlg}. 

Randomized local search heuristics are problem-independent algorithms that can be used as frameworks for the optimization of inherently difficult problems, such as combinatorial problems, graph problems, etc. 
We distinguish between Monte-Carlo algorithms and Las Vegas algorithms. 
Las Vegas algorithms are known to converge to the optimal solution, but their exact running time cannot be determined in advance. 
Monte-Carlo algorithms, on the other hand, have a fixed running time (usually measured by the number of iterations or the number of function evaluations performed), but we usually do not know the quality of the final output. 
The two threshold accepting algorithms presented next are Monte-Carlo algorithms for which the user may specify the number of iterations he is willing to invest for a good approximation of the star discrepancy value.
The genetic algorithm presented in Section~\ref{GenAlg}, on the other hand, is a Monte-Carlo algorithm with unpredictable running time (as we shall see below, in this algorithm, unconventionally, the computation is aborted when no improvement has happened for some $t$ iterations in a row). 

This said, it is clear that the lower bounds computed by both the threshold accepting algorithms as well as the one computed by the genetic algorithm may be arbitrarily bad. 
However, as all reported numerical experiments suggest, they are usually quite good approximations of the true discrepancy value---in almost all cases for which the correct discrepancy value can be computed the same value was also reported by the improved threshold accepting heuristic~\cite{GWW12} described below. 
We note that these heuristic approaches allow the computation of lower bounds for the star discrepancy also in those settings where the running time of exact algorithms like the one of Dobkin, Eppstein, and Mitchell described in Section~\ref{DEM} are not feasible.

\begin{algorithm2e}
\textbf{Initialization:}
Select $y \in \overline{\Gamma}(X)$ uniformly at random and compute 
$d^*_{\infty}(y,X)$\label{init}\;
\For{$i=1,2, \ldots, I$}{
\textbf{Mutation Step:}
Select a random neighbor $z$ of $y$ and compute $d^*_{\infty}(z,X)$\label{mutation}\;
\textbf{Selection step:}
\lIf{$d^*_{\infty}(z,X) - d^*_{\infty}(y,X) \geq T$}
{$y \leftarrow z$\label{selection}\;}
}
Output $d^*_{\infty}(y,X)$\;
\caption{Simplified scheme of a threshold accepting algorithm for the computation of star discrepancy values. 
$I$ is the runtime of the algorithm (number of iterations), and 
$T$ is the threshold value for the acceptance of a new candidate solution $z$.}
\label{alg:rsh}
\end{algorithm2e}

\emph{Threshold accepting} is based on a similar idea as the well-known 
simulated annealing algorithm \cite{KGV83}. 
In fact, it can be seen as a simulated annealing algorithm in which the selection step is derandomized, cf. Algorithm~\ref{alg:rsh} for the general scheme of a threshold accepting algorithm. 
In our application of computing star discrepancy values, we accept a new candidate solution $z$ if its local discrepancy $d^*_{\infty}(z,X)$ is not much worse than that of the previous step, and we discard $z$ otherwise. 
More precisely, we accept $z$ if and only if the difference $d^*_{\infty}(z,X) - d^*_{\infty}(y,X)$ is at least as large as some \emph{threshold value} $T$. 
The threshold value is a parameter to be specified by the user. We typically have $T<0$. $T<0$ is a reasonable choice as it prevents the algorithm from getting stuck in some local maximum of the local discrepancy function.
In the two threshold accepting algorithms presented below, 
$T$ will be updated frequently during the run of the algorithm (details follow).

The first to apply threshold accepting to the computation of star discrepancies were P.~Winker and K.-T.~Fang~\cite{WF97}.
Their algorithm was later improved in~\cite{GWW12}. 
In this section, we briefly present the original algorithm from~\cite{WF97}, followed by a short discussion of the modifications made in~\cite{GWW12}.

The algorithm of Winker and Fang uses the grid structure $\overline{\Gamma}(X)$. 
As in line~\ref{init} of Algorithm~\ref{alg:rsh}, they initialize the algorithm by selecting a grid point $y \in \overline{\Gamma}(X)$ uniformly at random.
In the mutation step (line~\ref{mutation}), a point $z$ is sampled uniformly at random from the neighborhood $\NN^{mc}_k(y)$ of $y$.
For the definition of $\NN^{mc}_k(y)$ let us first introduce the functions $\varphi_j$, $j \in \{1,\ldots,d\}$, which order the elements in $\overline{\Gamma}_j(X)$; i.e., 
for $n_j:=|\overline{\Gamma}_j(X)|$
the function $\varphi_j$ is a permutation of $\{1, \ldots, n_j\}$ with
$x^{(\varphi_j(1))}_j \leq \ldots \leq x^{(\varphi_j(n_j))}_j=1$.
For sampling a neighbor $z$ of $y$ we first draw $mc$ coordinates 
$j_1, \ldots, j_{mc}$ from $\{1, \ldots, d\}$ uniformly at random. 
We then select, independently and uniformly at random, for each $j_i$, $i=1, \ldots, mc$, a value 
$k_i \in \{-k, \ldots, -1, 0, 1, \ldots, k\}$. 
Finally, we let 
$z=(z_1, \ldots, z_d)$ with
\begin{align*}
z_j:=
\begin{cases}
y_j, &\text{ for } j \notin \{j_1, \ldots, j_{mc}\}\,,\\
y_j+k_j, &\text{ for } j \in \{j_1, \ldots, j_{mc}\}\,.
\end{cases}
\end{align*}
Both the values $mc \in \{1, \ldots, d\}$ and $k \in \{1, \ldots, n/2\}$ are inputs of the algorithm to be specified by the user. For example, if we choose $mc=3$ and $k=50$, then in the mutation step we change up to three coordinates of $y$, and for each such coordinate we allow to do up to $50$ steps on the grid $\overline{\Gamma}_j(X)$, either to the ``right'' or to the ``left''.

In the selection step (line~\ref{selection}), the search point $z$ is accepted if its discrepancy value is better than that of $y$ or if it is at least not worse than $d^*_{\infty}(y,X)+T$, for some threshold $T\leq 0$ that is determined in a precomputation step of the algorithm. 
Winker and Fang decided to keep the same threshold value for $\sqrt{I}$ iterations, and to replace it every $\sqrt{I}$ iterations with a new value $0 \geq T'>T$. 
The increasing sequence of threshold values guarantees that the algorithm has enough flexibility in the beginning to explore the search space, and enough stability towards its end so that it finally converges to a local maximum.
This is achieved by letting $T$ be very close to zero towards the end of the algorithm.

The algorithm by Winker and Fang performs well in numerical tests 
on rank-1 lattice rules, and it frequently computes the correct star discrepancy values in cases where this can be checked. 
However, as pointed out in~\cite{GWW12}, their algorithm does not perform very well in dimensions $10$ and larger. 
For this reason, a number of modifications have been introduced in~\cite{GWW12}. 
These modification also improve the performance of the algorithm in small dimensions.
 
The main differences of the algorithm presented in~\cite{GWW12} include a refined neighborhood structure that takes into account the topological structure of the point set $X$ and the usage of the concept of critical boxes as introduced in Definition~\ref{def:critical}. 
Besides this, there are few minor changes such as a variable size of the neighborhood structures and splitting the optimization process of $d^*_{\infty}(\cdot, X)$ into two separate processes for $\delta(\cdot, X)$ and $\overline{\delta}(\cdot,X)$, respectively.
Extensive numerical experiments are presented in~\cite{GWW12}. 
As mentioned above, in particular for large dimension this refined algorithm 
seems to compute better lower bounds for $d^*_{\infty}(\cdot, X)$ than the 
basic one from~\cite{WF97}.

We briefly describe the refined neighborhoods used in~\cite{GWW12}.
To this end, we first note that the neighborhoods used in Winker and Fang's algorithm do not take into account the absolute size of the gaps $x_j^{(\varphi(i+1))}-x_j^{(\varphi(i))}$ between two successive coordinates of grid points. This is unsatisfactory since large gaps usually indicate large differences in the local discrepancy function. 
Furthermore, for a grid cell $[y,z]$ in $\overline{\Gamma}(X)$
(i.e., 
$y,z\in \overline{\Gamma}(X)$ and 
$(z_j=x^{(\varphi_j(i_j))}_j ) \Rightarrow (y_j=x^{(\varphi_j(i_j+1))}_j )$
for all $j\in [d]$)
with large volume, we would expect that $\overline{\delta}(y,X)$ or $\delta(z,X)$ are also 
rather large.
For this reason, the following continuous neighborhood is considered. 
As in the algorithm by Winker and Fang we sample $mc$ coordinates $j_1, \ldots, j_{mc}$ from $\{1, \ldots, d\}$ uniformly at random.
The neighborhood of $y$ is the set 
$\NN_k^{mc}(y):=[\ell_1, u_1] \times \ldots \times [\ell_d, u_d]$ with 
\begin{align*}
[\ell_j, u_j]:=
\begin{cases}
\{y_j\}, & \text{ for } j \notin \{j_1, \ldots, j_{mc}\}\,,\\
[x^{(\varphi_j(\varphi_j^{-1}(y_j)-k\, \vee\, 1) )},
 x^{(\varphi_j(\varphi_j^{-1}(y_j)+k\,   \wedge\, n_j ))}], & \text{ for } j \in \{j_1, \ldots, j_{mc}\}\,,
\end{cases}
\end{align*}
where we abbreviate 
$\varphi_j^{-1}(y_j)-k\, \vee\, 1:=\max \{ \varphi_j^{-1}(y_j)-k, 1\}$ and, likewise, 
$\varphi_j^{-1}(y_j)+k\, \wedge\, n_j:=\min \{ \varphi_j^{-1}(y_j)+k, n_j\}$.
That is, for each of the coordinates $j \in \{j_1, \ldots, j_{mc}\}$ we do $k$ steps to the ``left'' and $k$ steps to the ``right''. 
We sample a point $\tilde{z} \in \NN_k^{mc}(y)$ (not uniformly, but according to some probability function described below) and we round $\tilde{z}$ once up and once down to the nearest critical box. For both these points $\tilde{z}^-$ and $\tilde{z}^+$ we compute the local discrepancy value, and we set as neighbor of $y$ the point 
$z \in \arg\max\{\overline{\delta}(\tilde{z}^-,X), \delta(\tilde{z}^+, X)\}$. 
The rounded grid points $\tilde{z}^+$ and $\tilde{z}^-$ are obtained by the \emph{snapping procedure} described in~\cite[Section 4.1]{GWW12}. 
We omit the details but mention briefly that rounding down to $\tilde{z}^-$ can be done deterministically ($\tilde{z}^-$ is unique), whereas for the upward rounding to $\tilde{z}^+$ there are several choices. The strategy proposed in~\cite{GWW12} is based on a randomized greedy approach.

We owe the reader the explanation of how to sample the point $\tilde{z}$.
To this end, we need to define the functions
\begin{equation*}
\Psi_j : 
[\ell_j, u_j] \to [0,1], 
r \mapsto \frac{r^d - (\ell_j)^d}{(u_j)^d - (\ell_j)^d}\,, 
j \in \{j_1, \ldots, j_{mc}\}
\end{equation*}
whose inverse functions are 
\begin{equation*}
\Psi^{-1}_j : [0,1] \to [\ell_j, u_j]\,, \, 
s\mapsto \Big( \big( (u_j)^d - (\ell_j)^d \big) s + (\ell_j)^d \Big)^{1/d}\,.
\end{equation*}
To sample $\tilde{z}$, we first sample values $s_1, \ldots, s_{mc} \in [0,1]$ independently and uniformly at random. We set 
$\tilde{z}_j:=\Psi^{-1}_j (s_j)$ for $j \in \{j_1, \ldots, j_{mc}\}$ and we set
$\tilde{z}_j:=y_j$ for $j \notin \{j_1, \ldots, j_{mc}\}$.
The intuition behind this probability measure is the fact that it favors larger coordinates than the uniform distribution. 
To make this precise, observe that in the case where $mc=d$, the probability measure on 
$\NN_k^{mc}(y)$ is induced by the affine transformation from 
$\NN_k^{mc}(y)$ to $[0,1]^d$ and the polynomial product measure
\begin{equation*}
\pi^d(\,dx) = \otimes^d_{j=1}  f(x_j)\,\lambda(\,dx_j) 
\text{ with density function }
f: [0,1] \to \R\,, \,r \mapsto d r^{d-1}
\end{equation*} on $[0,1]^d$.
The expected value of a point selected according to $\pi^d$ is $d/(d+1)$, whereas the expected value of a point selected according to the uniform measure (which implicitly is the one employed by Winker and Fang) has expected value $1/2$. Some theoretical and  experimental justifications for the choice of this probability measure are given in~\cite[Section 5.1]{GWW12}. 
The paper also contains numerical results for the computation of rank-1 lattice rules, Sobol' sequences, Faure sequences, and Halton sequences up to dimension $50$. The new algorithm based on threshold accepting outperforms all other algorithms that we are aware of. 
For more recent applications of this algorithm we refer the reader to Section~\ref{GENOPT}, where we present one example that indicates the future potential of this algorithm.

\subsubsection{An Approach Based on Genetic Algorithms}
\label{GenAlg}

A different randomized algorithm to calculate lower bounds for the star discrepancy of a given point set was proposed
by M.~Shah in~\cite{Sha10}. 
His algorithm is a \emph{genetic algorithm}.
Genetic algorithms are a class of local search heuristics that have been introduced in the sixties and seventies of the last century, cf.~\cite{Hol75} for the seminal work on evolutionary and genetic algorithms.
In the context of geometric discrepancies, genetic algorithms have also been successfully applied to the design of low-discrepancy sequences (cf. Section~\ref{GENOPT} for more details). 
In this section, we provide a very brief introduction into this class of algorithms, and we outline its future potential in the analysis of discrepancies.

While threshold accepting algorithms take their inspiration from physics, genetic algorithms are inspired by biology.
Unlike the algorithms presented in the previous section, 
in genetic algorithms, we typically do not keep only one solution candidate at a time, 
but we maintain a whole set of candidate solutions instead. 
This set is referred to as a \emph{population} in the genetic algorithms literature.
Algorithm~\ref{alg:ga} provides a high-level pseudo-code for genetic algorithms, adjusted again to the problem of computing lower bounds for the star discrepancy of a given point configuration. More precisely, this algorithm is a so-called $(\mu+\lambda)$ evolutionary algorithm (with $\lambda=C+M$ in this case). 
Evolutionary Algorithms are genetic algorithms that are based on Darwinian evolution principles. 
We discuss the features of such algorithms further below. 

\begin{algorithm2e}[h!]
\textbf{Initialization:}
Select 
$y^{(1)}, \ldots, y^{(\mu)} \in \overline{\Gamma}(X)$ uniformly at random and compute 
$d^*_{\infty}(y^{(1)},X), \ldots, d^*_{\infty}(y^{(\mu)},X)$\label{ga:init}\;
\For{$i=1,2, \ldots, I$}{
\textbf{Crossover Step:}
	\For{$j=1,2, \ldots, C$}{
		Select two individuals 
		$y, y' \in \{y^{(1)}, \ldots, y^{(\mu)}\}$ at random 
		and create from $y$ and $y'$ a new individual $z^{(j)}$ by recombination\;
		Compute $d^*_{\infty}(z^{(j)},X)$\label{ga:crossover}\;
		}
\textbf{Mutation Steps:}
	\For{$j=1,2, \ldots, M$}{
		Select an individual 
		$y \in \{y^{(1)}, \ldots, y^{(\mu)}, z^{(1)}, \ldots, z^{(C)}\}$ at random\;
		Sample a neighbor $n^{(j)}$ from $y$ and 
		compute $d^*_{\infty}(n^{(j)},X)$\label{ga:mutation}\;	
		}
\textbf{Selection step:}\\
\Indp\Indp
\label{ga:selection}
From $\{y^{(1)}, \ldots, y^{(\mu)}, z^{(1)}, \ldots, z^{(C)}, n^{(1)}, \ldots, n^{(M)}\}$ select---based on their local discrepancy values $d^*_{\infty}(\cdot,X)$---a subset of size $\mu$\;
Rename these individuals $y^{(1)}, \ldots, y^{(\mu)}$\;
\Indm\Indm
}
Output $d^*_{\infty}(y,X)$\;
\caption{Simplified scheme of a $(\mu+\lambda)$ evolutionary algorithm for the computation of star discrepancy values. 
$I$ is the runtime of the algorithm (i.e., the number of iterations), 
$C$ is the number of crossover steps per generation, and 
$M$ is the number of mutation steps.}
\label{alg:ga}
\end{algorithm2e}

As mentioned above, the nomenclature used in the genetic algorithms literature deviates from the standard one used in introductory books to algorithms. 
We briefly name a few differences. 
In a high-level overview, a genetic algorithm runs in several \emph{generations} (steps, iterations), 
in which the solution candidates (\emph{individuals}) from the current population are being \emph{recombined} and \emph{mutated}. 

We initialize such an algorithm by selecting $\mu$ individuals at random. They form the \emph{parent population} (line~\ref{ga:init}). 
To this population we first apply a series of \emph{crossover} steps (line~\ref{ga:crossover}), through which two (or more, depending on the implementation) individuals from the parent population are recombined. 
A very popular recombination operator is the so-called \emph{uniform crossover} through which two search points $y, y' \in \overline{\Gamma}(X)$ are recombined to some search point $z$ by setting $z_j:=y_j$ with probability $1/2$, and by setting $z_j=y'_j$ otherwise. Several other recombination operators exist, and they are often adjusted to the problem at hand. 
The random choice of the parents to be recombined must not be uniform, and it may very well depend on the local discrepancy values $d^*_{\infty}(y^{(1)},X), \ldots, d^*_{\infty}(y^{(\mu)},X)$, which are also referred to as the \emph{fitness} of these individuals.

Once $C$ such recombined individuals $z^{(1)}, \ldots, z^{(C)}$ have been created and evaluated, we enter the \emph{mutation} step, 
in which we compute for a number $M$ of search points one neighbor each, cf. line~\ref{ga:mutation}.  
Similarly as in the threshold accepting algorithm, Algorithm~\ref{alg:rsh}, it is crucial here to find a meaningful notion of neighborhood. This again depends on the particular application. For our problem of computing lower bounds for the star discrepancy value of a given point configuration, we have presented two possible neighborhood definitions in Section~\ref{threshold}. 
The newly sampled search points are evaluated, and from the set of old and new search points a new population is selected in the \emph{selection} step, line~\ref{ga:selection}. The selection typically depends again on the fitness values of the individuals. If always the $\mu$ search points of largest local discrepancy value are selected, we speak of an \emph{elitist selection scheme}. This is the most commonly used selection operator in practice. However, to maintain more diversity in the population, it may also be reasonable to use other selection schemes, or to randomize the decision.

In the scheme of Algorithm~\ref{alg:ga}, the algorithm runs for a fixed number $I$ of iterations. However, as we mentioned in the beginning of Section~\ref{threshold}, Shah's algorithm works slightly different. 
His algorithm stops when no improvement has happened for some $t$ iterations in a row, where $t$ is a parameter to be set by the user.

The details of Shah's implementation can be found in~\cite{Sha10}. 
His algorithm was used in~\cite{OSG12} for the approximation of the star discrepancy value of 10-dimensional permuted Halton sequences. 
Further numerical results are presented in~\cite{Sha10}.
The results reported in~\cite{Sha10}, however, are not demanding enough to make a proper comparison between his algorithm and the ones presented in the previous section.
On the few instances where a comparison seems meaningful, the results based on the threshold accepting algorithms outperform the ones of the genetic algorithm, cf.~\cite[Section 6.4]{GWW12} for the numerical results. 
Nevertheless, it seems that the computation of star discrepancy values with genetic and 
evolutionary algorithms is a promising direction and further research would be of interest.

\subsubsection*{Notes}
In the literature one can find some attempts to compute for $L_\infty$-discrepancies
the smallest possible discrepancy value of all $n$-point configurations. 
For the star discrepancy  B.~White determined in \cite{Whi77} the smallest possible
discrepancy values for $n = 1,2,\ldots,6$ in dimension $d=2$, and T.~Pillards, B.~Vandewoestyne,
and R.~Cools in \cite{PVC06} for $n=1$ in arbitrary dimension $d$.
G.~Larcher and F.~Pillichshammer provided in \cite{LP07} for the star and the extreme discrepancy
the smallest discrepancy values for $n=2$ in arbitrary dimension $d$.
Furthermore, they derived for the isotrope discrepancy the smallest value for $n=3$ in dimension
$d=2$ and presented good bounds for the smallest value for $n=d+1$ in arbitrary dimension
$d\ge 3$. (Note that the isotrope discrepancy of $n<d+1$ points in dimension $d$ is necessarily
the worst possible discrepancy $1$.)

\section{Calculation of $L_p$-Discrepancies for $p\notin \{2,\infty\}$}
\label{L_P}

This section is the shortest section in this book chapter. The reason for this is not that the computation
of $L_p$-discrepancies, $p\notin \{2,\infty\}$, is an easy task which is quickly explained, but rather that not much work has been done so far and that therefore, unfortunately, 
not much is known to date.
We present here a generalization of Warnock's formula for even $p$. 

Let $\gamma_1 \ge \gamma_2 \ge \cdots \ge \gamma_d \ge 0$, and let $(\gamma_u)_{u\subseteq \{1,\ldots,d\}}$ be the corresponding
product weights; i.e., $\gamma_u = \prod_{j\in u} \gamma_j$ for all $u$. For this type of weights 
G.~Leobacher and F.~Pillichshammer derived  a formula for the weighted $L_p$-star discrepancy $d^*_{p, \gamma}$ for arbitrary even positive integers $p$
that generalizes the formula (\ref{war_weight}):
\begin{equation}\label{leo_pil}
\begin{split}
 &(d^*_{p,\gamma}(X))^p = \\
&\sum^p_{\ell = 0} \binom{p}{\ell} 
\left( - \frac{1}{n} \right)^{\ell} 
\sum_{(i_1,\ldots,i_{\ell}) \in \{1,\ldots,n\}^{\ell}} \prod^d_{j = 1} 
\left( 1 + \gamma_j \frac{1-\max_{1\le k \le \ell} (x^{(i_k)}_j)^{p-\ell +1}}{p-\ell +1} \right),
\end{split}
\end{equation}
see \cite[Thm.~2.1]{LP03} (notice that in their definition of the weighted discrepancy they replaced the weights $\gamma_u$ appearing in our definition (\ref{weighted_disc}) 
by $\gamma_u^{p/2}$). Recall that in the special case where $\gamma_1 = \gamma_2 = \cdots = \gamma_d = 1$, the weighted $L_p$-star discrepancy $d^*_{p,\gamma}(X)$
coincides with Hickernell's modified $L_p$-discrepancy. Using the approach from \cite{LP03} one may derive analogous formulas for the $L_p$-star and $L_p$-extreme discrepancy.
The formula (\ref{leo_pil}) can be evaluated directly at cost $O(d n^p)$, where the implicit constant in the big-$O$-notation depends on $p$. Obviously, the computational
burden will become infeasible even for moderate values of $n$ if $p$ is very large. 

Apart from the result in \cite{LP03} we are not aware of any further results that are helpful for the calculation or approximation of weighted $L_p$-discrepancies. 

\subsubsection*{Notes}
The calculation of \emph{average $L_p$-discrepancies} of Monte Carlo point sets attracted 
reasonable attention in the literature. One reason for this is that  $L_p$-discrepancies
such as, e.g., the $L_p$-star or $L_p$-extreme discrepancy,
converge to the corresponding $L_\infty$-discrepancy if $p$ tends to infinity, see, e.g., \cite{Gne05, HNWW01}. Thus the $L_p$-discrepancies 
can be used to derive results for the corresponding $L_\infty$-discrepancy.
In the literature one can find explicit representations of  average $L_p$-discrepancies in terms of sums involving Stirling numbers of the first 
and second kind as well as upper bounds and formulas for their asymptotic behavior, see, e.g., \cite{Gne05, HNWW01, HW12, LP03, Stei10}.

\section{Some Applications}
\label{APPLICATIONS}

We present some applications of algorithms that approximate discrepancy measures.
The aim is to show here some more recent examples of how the algorithms are used in 
practice, what typical instances are, and what kind of problems occur. Of course, the selection of topics
reflects the interest of the authors and is far from being complete. Further
applications can, e.g., be found in the design of computer experiments (``experimental
design''), see \cite{FLWZ00, Lem09}, the generation of pseudo-random numbers,
see \cite{Knu81, LH98, Nie92, Tez95}, or in computer graphics, see \cite{DEM96}. 

\subsection{Quality Testing of Point Sets}
\label{QUALTEST}

A rather obvious application of discrepancy approximation algorithms is to
estimate the quality of low-discrepancy point sets or, more generally, deterministic or randomized quasi-Monte Carlo point configurations.

Thi\'emard, e.g., used his algorithm from \cite{Thi01a}, which we described in Section~\ref{bracketing},
to provide upper and lower bounds
for the star discrepancy of Faure $(0,m,s)$-nets \cite{Fau82} with sample sizes varying from $1,048,576$ points in the smallest dimension $d=2$ to $101$ points in the largest dimension $100$ (where, not very surprisingly, the resulting discrepancy is almost $1$).
He also uses his algorithm to compare the performance of two sequences of pseudo-random numbers, generated by Rand() and MRG32k3a \cite{LEc99}, and Faure, Halton \cite{Hal60}, and 
Sobol' \cite{Sob67} sequences by calculating bounds for their star discrepancy for sample
sizes between $30$ and $250$ points in dimension $7$. 

For the same instances Thi\'emard was able to calculate the exact star discrepancy of the Faure, Halton and Sobol' sequences by using his algorithm from \cite{Thi01b}, which we described in Section~\ref{ILP}, see \cite[Sect.~4.3]{Thi01b}. 

In the same paper he provided the exact star discrepancy of Faure $(0,m,s)$-nets ranging from sample sizes of $625$ points in dimension $4$ to $169$ points in dimension $12$.
These results complement the computational results he achieved for (less demanding)
instances in \cite{Thi00} with the help of the algorithm presented there.

Algorithms to approximate discrepancy measures were also used to judge the 
quality of different types of generalized Halton sequences. Since these sequences are also
important for our explanation in Section~\ref{GENOPT}, we give a definition here.

Halton sequences are a generalization of the one-dimensional van der Corput sequences. 
For a prime base $p$ and a positive integer $i \in \N$ let 
$i=d_{k}d_{k-1}\ldots d_2d_1$ be the digital expansion of $i$ in base $p$. That is, let $d_1, \ldots, d_k$ be such that $i=\sum_{\ell=1}^k{d_{\ell}p^{\ell-1}}$. 
Define the \emph{radical inverse function} $\phi_p$ in base $p$ by
\begin{equation}\label{rad_inv}
\phi_p(i):=\sum_{\ell=1}^k{d_{\ell}p^{-\ell}}\,.
\end{equation}
Let $p_j$ denote the $j$th prime number. 
The $i$th element of the $d$-dimensional \emph{Halton sequence} is defined as
\begin{align*}
x^{(i)}:=(\phi_{p_1}(i), \ldots, \phi_{p_d}(i))\,.
\end{align*}
The Halton sequence is a low-discrepancy sequence, i.e., its first $n$ points
$X=(x^{(i)})^n_{i=1}$ in dimension $d$ satisfy the star discrepancy bound
\begin{equation}\label{halton_bound}
 d^*_\infty(X) = O \big( n^{-1}\,\ln(n)^d \big).
\end{equation}
In fact, the Halton sequence was the first construction for which (\ref{halton_bound}) 
was verified for any dimension $d$ \cite{Hal60}, and up to now there is no sequence
known that exhibits a better asymptotical behavior than (\ref{halton_bound}).

\begin{figure}
\includegraphics[scale=0.9]{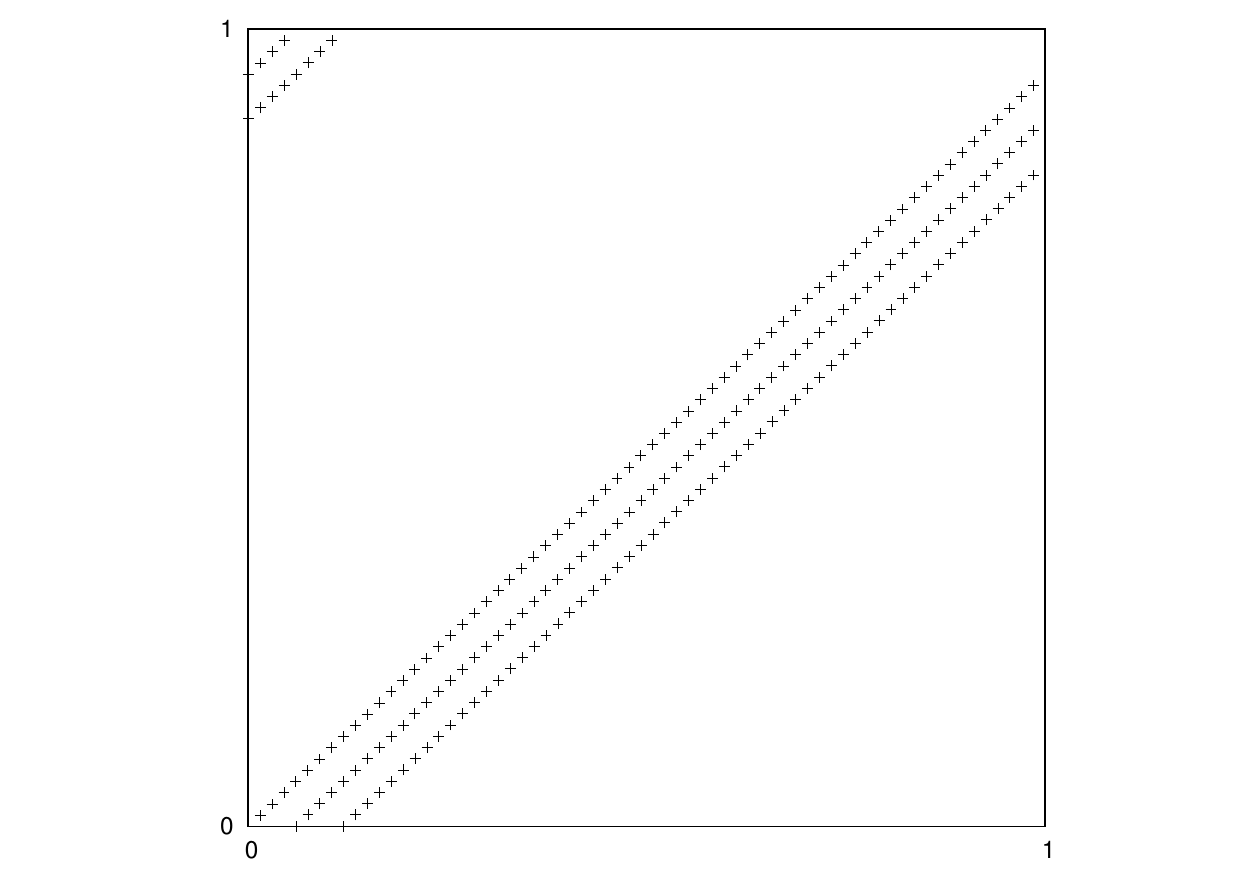}
\caption{The first 200 points of the 20-dimensional Halton sequence, projected to
  dimensions 19 and 20.}
\label{fig:halton}
\end{figure}

Nevertheless, for higher dimension $d$ the use of larger prime $p_j$ bases leads to 
some irregularity phenomenon, which is often referred to as 
\emph{high correlation between higher bases} \cite{BW79}. 
This phenomenon can easily be visualized by looking at two-dimensional projections
of Halton sequences over higher coordinates, see Figure~\ref{fig:halton}.

To reduce these undesirable effects, Braaten and Weller \cite{BW79} suggested to 
use generalized Halton sequences.  To obtain such a sequence, one applies digit 
permutations to the Halton sequence:
For a permutation $\pi_p$ of $\{0,1,\ldots,p-1\}$ with fixpoint $\pi_p(0)=0$
define in analogy to (\ref{rad_inv}) the \emph{scrambled radical inverse function}
$\phi^{\pi}_p$ by
\begin{equation*}
\phi^{\pi}_p(i) :=\sum_{\ell=1}^k{\pi_p(d_{\ell})p^{-\ell}}.
\end{equation*}
The $i$th element of the \emph{generalized} (or \emph{scrambled}) \emph{Halton sequence} 
in $d$ dimensions is then defined by
\begin{align}
\label{generalHalton}
x^{(i)}(\Pi):=(\phi^{\pi_{p_1}}_{p_1}(i), \ldots, \phi^{\pi_{p_d}}_{p_d}(i))\,,
\end{align}
where we abbreviate $\Pi:=(\pi_{p_1}, \ldots, \pi_{p_d})$.

In several publications different permutations were proposed to shuffle the digits of 
the Halton sequence, see, e.g., \cite{Ata04, CMW05, DGTL12, FL09, OSG12, VC06, WH00}
and the literature mentioned therein.
Many authors tried to compare the quality of some of these permutations by considering
theoretical or numerical discrepancy bounds or other numerical tests.

Vandewoestyne and Cools \cite{VC06}, e.g., calculated the $L_2$-star and $L_2$-extreme
discrepancy of several generalized Halton sequences by using Warnock's formula (\ref{warnock}) 
and the corresponding modification (\ref{war_extreme}), respectively.
The instances they studied ranged from $10,000$ sample points in dimension $8$
to $1,000$ points in dimension $64$. They reported that those generalized Halton sequences
performed best which are induced by the simple \emph{reverse permutations}
$\pi_p(0)=0$ and $\pi_p(j)=p-j$ for $j=1,\ldots,p-1$. These sequences
showed usually a smaller $L_2$-star and $L_2$-extreme discrepancy as, e.g., 
the original Halton sequences, the generalized Halton sequences proposed in
\cite{ Ata04,BW79} or the randomized Halton sequences from \cite{WH00}; for more
details see \cite{VC06}.

\"Okten, Shah, and Goncharov \cite{OSG12} tried to compare different generalized Halton sequences
by calculating bounds for their $L_\infty$-star discrepancy. For the calculations of the
upper bounds for the star discrepancy they used Thi\'emard's algorithm presented in
\cite{Thi01a}. For the calculation of the lower bounds they used Shah's algorithm
\cite{Sha10}, which gave consistently better lower bounds than Thi\'emard's algorithm.
They did the calculation for instances of $100$ points in dimension $5$ and $10$
and studied different cases of prime bases. So they considered, e.g., different generalized Halton sequences in $d=5$ where the prime bases are the $46$th to $50$th
prime numbers $p_{46},\ldots,p_{50}$ (which corresponds to the study of the projections
over the last $5$ coordinates of $50$-dimensional sequences induced by the first $50$ primes).
\"Okten et al. found that, apart from the original Halton sequence, the reverse
permutations led to the highest discrepancy bounds, in contrast to their very low
$L_2$-star and $L_2$-extreme discrepancy values. (This indicates again that the conventional
$L_2$-star and $L_2$-extreme discrepancy are of limited use for judging the uniformity of point sets,
cf. also the notes at the end of Section~\ref{Ltwo}.)
Furthermore, they reported that the 
average star discrepancy bounds of generalized Halton sequences induced by random
permutations were rather low. For more details we refer to \cite{OSG12}.

Algorithms to approximate star discrepancies were used in \cite{DGW10} to compare
the quality of small samples of classical low-discrepancy points and Monte Carlo
points with new point sets generated in a randomized and in a deterministic fashion.
Here, ``small samples'' has to be understood as point sets whose size is small compared
to the dimension, say, bounded by a constant times $d$ or $d^{3/2}$. Notice that for this sample sizes asymptotic bounds like (\ref{halton_bound}) do not provide any helpful 
information. The classical point sets were Halton-Hammersley points \cite{DP10, Hal60, Nie92}, Sobol' points \cite{JK08, Sob67, SAKK11}, and Faure points shuffled by a Gray code
\cite{Fau82,Tez95, Thi98}. The algorithms that generated the new point sets rely on certain random experiments based on randomized rounding with hard constraints \cite{Doe06, Sri01}
and large deviation bounds that guarantee small discrepancies with high probability.
The deterministic versions of these algorithms make use of derandomization techniques
from \cite{Doe06,DW09}. The concrete algorithms are randomized versions of the component-by-component (CBC) construction proposed in \cite{DGKP08} and implemented in \cite{DGW09} and a randomized and derandomized version of the algorithm proposed in
\cite{DG08, DGW10}. The CBC construction has the advantage that it is faster and can
be used to extend given low-discrepancy point sets in the dimension, but its theoretical
error bound is worse than the one for the latter algorithm.

In the numerical experiments different methods to approximate the discrepancy were used, depending on the instance sizes. For instances of $145 - 155$ points in dimension $7$ and 
$85 - 95$ points in dimension $9$ exact discrepancy calculations were performed by using
the method of Bundschuh and Zhu \cite{BZ93}. For larger instances this was infeasible, so
for instances ranging from $145 - 155$ points in dimension $9$ to $65 - 75$ points in dimension $12$ a variant of the algorithm of Dobkin, Eppstein, and Mitchell \cite{DEM96} was used that gained speed by allowing for an imprecision of order $d/n$. 
It should be noted that since the publication of these experiments, the implementation of
this algorithm has been improved; a version that attains the same speed without making
the $d/n$-order sacrifice of precision is available at the third author's homepage at
\texttt{http://www.mpi-inf.mpg.de/$\sim$wahl/}. 
For the final set of instances, ranging from $145 - 155$ points in dimension $12$ to
$95 - 105$ points in dimension $21$, the authors relied on upper bounds from Thi\'emard's algorithm from \cite{Thi01a} and on lower bounds from a randomized algorithm based on threshold accepting from \cite{Win07}, which is a precursor of the algorithm from \cite{GWW12}. 
Similarly as in \cite{OSG12}, the randomized algorithm led consistently to better lower
bounds than Thi\'emard's algorithm. For the final instances of $95 - 105$ points in dimension $21$ the gaps between the upper and lower discrepancy bounds were roughly
of the size $0.3$, thus these results can only be taken as very coarse indicators for the quality of the considered point sets.

As expected, the experiments indicate that for small dimension and relatively large sample
sizes the Faure and Sobol' points are superior, but the derandomized algorithms (and also
their randomized versions) performed better than Monte Carlo, Halton-Hammersley, and
Faure points in higher dimension for relatively small sample sizes. From the classical 
low-discrepancy point sets the Sobol' points performed best and were competitive 
for all instances. For more details see \cite[Sect.~4]{DGW10}.

\subsection{Generating Low-Discrepancy Points via an Optimization Approach}
\label{GENOPT}

In Section~\ref{GenAlg} we have seen that genetic algorithms can be used for an approximation of the star discrepancy values of a given point configuration. 
Here in this section we present another interesting application of biology-inspired algorithms in the field of geometric discrepancies. 

Whereas the works presented in Sections~\ref{approximation} focus mainly on the computation of star discrepancy values, the authors of~\cite{DGTL12}, 
F.-M. de Rainville, C. Gagn\'e, O. Teytaud, and D. Laurendeau, apply evolutionary algorithms (cf. Section~\ref{GenAlg} for a brief introduction)
to generate \emph{low discrepancy point configurations.} 
Since the fastest known algorithms to compute the exact star discrepancy values have running time exponential in the dimension (cf. Section~\ref{DEM} for a discussion of this algorithm and Section~\ref{Complexity} for related complexity-theoretic results), and the authors were---naturally---interested in a fast computation of the results (cf.~\cite[Page 3]{DGTL12}), the point configurations considered in~\cite{DGTL12} are optimized for Hickernell's modified $L_2$-star discrepancies, see equation (\ref{hickernell_disc}) in the introduction to this chapter. 
As mentioned in Section~\ref{Ltwo}, $L_2$-star discrepancies can be computed efficiently via Warnock's formula, cf. equation~(\ref{warnock}).
Similarly, Hickernell's modified $L_2$-discrepancy can be computed efficiently with $O(dn^2)$ arithmetic operations. 

The point sets generated by the algorithm of de Rainville, Gagn\'e, Teytaud, and Laurendeau are generalized Halton sequences.
As explained in Section~\ref{QUALTEST}, generalized Halton sequences are digit-permuted versions of the Halton sequence, cf. equation (\ref{generalHalton}).
The aim of their work is thus to find permutations $\pi_{p_1}, \ldots, \pi_{p_d}$ such that the induced generalized Halton sequence has small modified $L_2$-discrepancy.
They present numerical results for sequences with $2,500$ points in dimensions 20, 50, and 100 and they compare the result of their evolutionary algorithm with that of standard algorithms from the literature. 
Both for the modified $L_2$-discrepancy as well as for the $L_2$-star discrepancy the results indicate that the evolutionary algorithm is at least comparable, if not superior, to classic approaches.

The evolutionary algorithm of de Rainville, Gagn\'e, Teytaud, and Laurendeau uses both the concept of crossover and mutation (as mentioned above, see Section~\ref{GenAlg} for a brief introduction into the notations of genetic and evolutionary algorithms). 
The search points are permutations, or, to be more precise configurations of permutations. 
The mutation operator shuffles the values of the permutation by deciding independently for each value in the domain of the permutation if its current value shall be swapped, and if so, with which value to swap it. 
The crossover operator recombines two permutations by iteratively swapping  pairs of values in them. 

We recall that in the definition of the generalized Halton sequence, equation (\ref{generalHalton}), we need a vector (a \emph{configuration}) $\Pi=(\pi_{p_1}, \ldots, \pi_{p_d})$ of permutations of different length. 
The length of permutation $\pi_{p_i}$ is determined by the $i$th prime number $p_i$. 
These configurations are extended component-by-component. 
That is, the evolutionary algorithm first computes an optimal permutation $\pi_2$ on $\{0,1\}$ and sets $\Pi:=(\pi_2)$. 
It then adds to $\Pi$, one after the other, 
the optimal permutation $\pi_3$ for $\{0,1,2\}$, the one $\pi_5$ on $\{0,1,2,3,4\}$, and so on. 
Here, ``optimal'' is measured with respect to the fitness function, which is simply the modified (Hickernell) $L_2$-discrepancy of the point set induced by $\Pi$. We recall from equation (\ref{generalHalton}) that a vector $\Pi$ of permutations fully determines the resulting generalized Halton sequence.

The iterative construction of $\Pi$ allows the user to use, for any two positive integers $d<D$, the same permutations $\pi_{p_1},\ldots,\pi_{p_d}$ in the first $d$ dimensions of the two sequences in dimensions $d$ and $D$, respectively.
 
The population sizes vary from 500 individuals for the 20-dimensional point sets to 750 individuals for the 50- and 100-dimensional point configurations. 
Similarly, the number of generations that the author allow the algorithm for finding the best suitable permutation ranges from 500 in the 20-dimensional case to 750 for the 50-dimensional one, and 1,000 for the 100-dimensional configuration.  

It seems evident that combining the evolutionary algorithm of de Rainville, Gagn\'e, Teytaud, and Laurendeau with the approximation algorithms presented in Section~\ref{approximation}
is a promising idea to generate low star-discrepancy point sequences. This is ongoing work of Doerr and de Rainville~\cite{DW12}. 
In their work, Doerr and de Rainville use the threshold accepting algorithm from~\cite{GWW12} (cf. Section~\ref{threshold}) for the intermediate evaluation of the candidate permutations. 
Only the final evaluation of the resulting point set is done by the exact algorithm of Dobkin, Eppstein, and Mitchell~\cite{DEM96} (cf. Section~\ref{DEM}).
This allows the computation of low star-discrepancy sequences in dimensions where an exact evaluation of the intermediate permutations of the generalized Halton sequences would be far too costly.
Using this approach, one can, at the moment, not hope to get results for such dimensions where the algorithm of Dobkin, Eppstein, and Mitchell does not allow for a final evaluation of the point configurations. However, due to the running time savings during the optimization process, one may hope to get good results for moderate dimensions for up to, say, 12 or 13. 
Furthermore, it seems possible to get a reasonable indication of good generating permutations in (\ref{generalHalton}) for dimensions much beyond this, if one uses only the threshold accepting algorithm to guide the search. As is the case for all applications of the approximation algorithms presented in Section~\ref{approximation}, in such dimensions, however, we do not have a \emph{proof} that the computed lower bounds are close to the exact discrepancy values.

We conclude this section by emphasizing that, in contrast to the previous subsection, where 
the discrepancy approximation algorithms were only used to \emph{compare} the quality 
of different point sets, here in this application the calculation of the discrepancy is an integral part of the optimization
process to \emph{generate} good low discrepancy point sets.

\subsubsection*{Notes} 
Also other researchers used algorithms to approximate discrepancy measures in combination with optimization 
approaches to generate low discrepancy point sets. A common approach described in \cite{DEM96} is to 
apply a multi-dimensional optimization algorithm  repeatedly to randomly jiggled versions of low discrepancy
point sets to search for smaller local minima of the discrepancy function.
In \cite{FLWZ00, LSW10} the authors used the optimization heuristic threshold accepting to generate experimental
designs with small discrepancy. The discrepancy measures considered in \cite{FLWZ00} are the $L_\infty$- and 
the $L_2$-star discrepancy, as well as the centered, the symmetric, and the modified $L_2$-discrepancy.
The discrepancy measure considered in \cite{LSW10} is the central composite discrepancy for a flexible region
as defined in \cite{CH10}.

\subsection{Scenario Reduction in Stochastic Programming}
\label{SCENARIO}

Another interesting application of discrepancy theory is in the area of \emph{scenario reduction}. We briefly describe the underlying problem, and we illustrate the importance of discrepancy computation in this context. 

Many real-world optimization problems, e.g., of financial nature, are subject to stochastic uncertainty. 
\emph{Stochastic programming} is a suitable tool to model such problems. 
That is, in stochastic programming we aim at minimizing or maximizing the expected value of a random process, typically taking into account several constraints. 
These problems, of both continuous and discrete nature, are often infeasible to solve optimally. 
Hence, to deal with such problems in practice, one often resorts to approximating the underlying (continuous or discrete) probability distribution by a discrete one of much smaller support. 
In the literature, this approach is often referred to as \emph{scenario reduction}, cf.~\cite{HKR09, Rom09} and the numerous references therein.

For most real-world situations, the size of the support of the approximating probability distribution has a very strong influence on (i.e., ``determines'') the complexity of solving the resulting stochastic program. 
On the other hand, it is also true for all relevant similarity measures between the original and the approximating probability distribution that a larger support of the approximating distribution allows for a better approximation of the original one. 
Therefore, we have a fundamental trade-off between the running time of the stochastic program and the quality of the approximation. 
This trade-off has to be taken care of by the practitioner, and his decision typically depends on his time constraints and the availability of computational resources. 

To explain the scenario reduction problem more formally, we need to define a suitable measure of similarity between two distributions. 
For simplicity, we regard here only distance measures for two discrete distributions whose support is a subset of $[0,1]^d$. 
Unsurprisingly, a measure often regarded in the scenario reduction literature is based on discrepancies, cf. again~\cite{HKR09, Rom09} and references therein. 
Formally, let $P$ and $Q$ be two discrete Borel distributions on $[0,1]^d$ and let $\BB$ be a system of Borel sets of $[0,1]^d$. 
The $\BB$-discrepancy between $P$ and $Q$ is 
\begin{align*} 
\disc_\infty(\BB;P,Q)= \sup_{B \in \BB} |P(B) - Q(B)|\,.
\end{align*}
The right choice of the set $\BB$ of test sets depends on the particular application. Common choices for $\BB$ are 
\begin{itemize}
\item $\CC_d$, the class of all axis-parallel
half-open boxes,
\item $\RR_d$, the class of all half-open axis-parallel boxes, 
\item $\PP_{d,k}$, the class of all polyhedra having at most $k$ vertices, and
\item $\II_d$, the set of all closed, convex subsets,
\end{itemize}
which were introduced in Section~\ref{INTRO}. 

In the scenario reduction literature, the discrepancy measure associated with $\CC_d$ is referred to as star discrepancy, \emph{uniform}, or \emph{Kolmogorov metric};
the one associated with $\II_d$ is called \emph{isotrope discrepancy}, whereas
the distance measure induced by 
$\RR_d$ is simply called the \emph{rectangle discrepancy} measure, and the one induced by $\PP_{d,k}$ as \emph{polyhedral discrepancy}.

With these distance measures at hand, we can now describe the resulting approximation problem. 
For a given distribution $P$ of support $\{x^{(1)}, \ldots, x^{(N)}\}$ we aim at finding, for a given integer $n<N$, a distribution $Q$ such that 
(i) the support $\{y^{(1)}, \ldots, y^{(n)} \}$ of $Q$ is a subset of the support of $P$ of size $n$ and that 
(ii) $Q$ minimizes the distance between $P$ and $Q$. 

Letting 
$\delta(z)$ denote the Dirac measure placing mass one at point $z$, 
$p^{(i)}:=P(x^{(i)})$, $1 \leq i \leq N$,
and 
$q^{(i)}:=P(x^{(i)})$, $1 \leq i \leq n$, we can formulate this minimization problem as
\begin{align}
\label{eq:scenario}
& \text{minimize } \quad
	\disc_\infty(\BB;P,Q)=\disc_\infty(\BB; \sum_{i=1}^N{p^{(i)} \delta(x^{(i)})}, 	
											 \sum_{i=1}^n{q^{(i)} \delta(y^{(i)})})\\
&\nonumber  \text{subject to }
\quad
\{y^{(1)}, \ldots, y^{(n)} \}  \subseteq \{x^{(1)}, \ldots, x^{(N)} \}\,,	\\		&\nonumber \quad\quad\quad\quad\quad\quad\quad\quad\quad\quad q^{(i)} >0, \quad 1 \leq i \leq n\,, \\
&\nonumber \quad\quad\quad\quad\quad\quad\quad\quad\quad \sum_{i=1}^n{q^{(i)}}=1.			 
\end{align}
This optimization problem can be decomposed into an \emph{outer} optimization problem of finding the support $\{y^{(1)}, \ldots, y^{(n)} \}$ and an \emph{inner} optimization problem of finding---for fixed support $\{y^{(1)}, \ldots, y^{(n)} \}$---the optimal probabilities $q^{(1)}, \ldots, q^{(n)}$.

Some heuristics for solving both the inner and the outer optimization problems have been suggested in~\cite{HKR09}. In that paper, Henrion, K\"uchler, and R\"omisch mainly regard the star discrepancy measure, but results for other distance measures are provided as well; see also the paper~\cite{HKR08} by the same set of authors for results on minimizing the distance with respect to polyhedral discrepancies. 

We present here a few details about the algorithms that Henrion, K\"uchler, and R\"omisch developed for the optimization problem (\ref{eq:scenario}) with respect to the star discrepancy. A more detailed description can be found in their paper~\cite{HKR09}. 

Two simple heuristics for the outer optimization problem are \emph{forward} and \emph{backward selection.}
In forward selection we start with an empty support set $Y$, and we add to $Y$, one after the other, the element from $X=\{x^{(1)}, \ldots, x^{(N)}\}$ that minimizes the star discrepancy between $P$ and $Q$, for an optimal allocation $Q$ of probabilities $q^{(1)}, \ldots, q^{(|Y|)}$ to the points in $Y$. We stop this forward selection when $Y$ has reached its desired size, i.e., when $|Y|=n$.

Backward selection follows an orthogonal idea. We start with the full support set $Y=X$ and we remove from $Y$, one after the other, the element such that an optimal probability distribution $Q$ on the remaining points of $Y$ minimizes the star discrepancy $\disc_\infty(\CC_d;P,Q)$. Again we stop once $|Y|=n$.
It seems natural that forward selection is favorable for values $n$ that are much smaller than $N$, and, likewise, backward selection is more efficient when the difference $N-n$ is small. 

For the inner optimization problem of determining the probability distribution $Q$ for a fixed support $Y$, Henrion, K\"uchler, and R\"omisch formulate a linear optimization problem.
Interestingly, independently of the discrepancy community, the authors develop to this end the concept of \emph{supporting} boxes---a concept that coincides with the critical boxes introduced in Definition~\ref{def:critical}, Section~\ref{EAC}. 
Using these supporting boxes, they obtain a linear program that has much less constraints than the natural straightforward formulation resulting from problem (\ref{eq:scenario}).
However, the authors remark that not the solution to the reduced linear program itself is computationally challenging, but the computation of the supporting (i.e., critical) boxes. 
Thus, despite significantly reducing the size of the original problem (\ref{eq:scenario}), introducing the concept of critical boxes alone is not sufficient to considerably reduce the computational effort required to solve problem (\ref{eq:scenario}). 
The problems considered in~\cite{HKR09} are thus only of moderate dimension and moderate values of $N$ and $n$. 
More precisely, results for four and eight dimensions with $N=100$, $N=200$, and $N=300$ points are computed. The reduced scenarios have $n=10$, $20$, and $30$ points. 

Since most of the running time is caused by the computation of the star discrepancy values, it seems thus promising to follow a similar approach as in Section~\ref{GENOPT} and to use one of the heuristic approaches for star discrepancy estimation presented in Section~\ref{approximation} for an intermediate evaluation of candidate probability distributions $Q$. This would allow us to compute the exact distance of $P$ and $Q$ only for the resulting approximative distribution $Q$. 

We conclude this section by mentioning that similar approaches could be useful also for the other discrepancy measures, e.g., the rectangle or the isotrope discrepancy of $P$ and $Q$. However, for this to materialize, more research is needed to develop good approximations of such discrepancy values. (This is particularly the case for the isotrope discrepancy---see the comment regarding the isotrope discrepancy at the beginning of Section~\ref{L_INFTY}.) 


\begin{acknowledgement}
The authors would like to thank Sergei Kucherenko, Shu Tezuka, Tony Warnock, Greg Wasilkowski, Peter Winker, and an 
anonymous referee for their valuable comments.

Carola Doerr is supported by a Feodor Lynen postdoctoral research fellowship of the Alexander von Humboldt Foundation and by the Agence Nationale de la Recherche under the project ANR-09-JCJC-0067-01.

The work of Michael Gnewuch was supported by the German Science Foundation DFG
under grant GN-91/3 and the Australian Research Council ARC.

The work of Magnus Wahlstr\"om was supported by the German Science Foundation DFG via its
priority program ``SPP 1307: Algorithm Engineering'' under grant DO 749/4-1.
\end{acknowledgement}

%
%
%

\bibliographystyle{aps-nameyear}
\bibliography{doerr_gnewuch_wahlstroem_references}

\end{document}